\documentclass[hidelinks, 12pt, oneside,reqno]{amsart} 
\usepackage{geometry}
\usepackage{dsliheader}

\begin{document}

\title[Mass Drop and Multiplicity]{Mass Drop and Multiplicity in Mean Curvature Flow}

\author{Alec Payne}
\address{Courant Institute, New York University, New York City, NY 10012}
\email{ajp697@nyu.edu}

\date{}

\begin{abstract}  
Brakke flow is defined with a variational inequality, which means it may have discontinuous mass over time, i.e.\ have mass drop. It has long been conjectured that the Brakke flow associated to a nonfattening level set flow has no mass drop and achieves equality in the Brakke inequality. Under natural assumptions, we show that a flow has no mass drop if and only if it satisfies the multiplicity one conjecture $\mathcal{H}^n$-a.e. One application is that there is no mass drop for level set flows with mean convex neighborhoods of singularities, and a generic flow has no mass drop until there is a higher multiplicity planar tangent flow. Also, if a nonfattening flow has no higher multiplicity planes as limit flows, then each limit flow has no mass drop. We upgrade these results to equality in the Brakke inequality for certain important cases. We show that nonfattening flows with three-convex blow-up type are Brakke flows with equality. This includes flows with generic singularities in dimension three and flows with mean convex neighborhoods of singularities in dimension four.
\end{abstract}
\maketitle
\vspace{-.2in}
\section{Introduction}

Mean curvature flow is the motion of hypersurfaces with normal velocity equal to the mean curvature. The flow is poorly understood for an arbitrary initial condition, even in $\bR^3$. Thus far, it is only tractable under assumptions on the curvature or the entropy. 

Mean curvature flow holds promise for geometric and topological applications. Some recent applications include classifying two-convex hypersurfaces~\cite{HuiskenSinestrari09, HaslhoferKleiner172, BrendleHuisken16, BrendleHuisken18}, classifying low entropy surfaces~\cite{BernsteinWang17, BernsteinWang18, coldingilmanenminicozziwhite13, mramorwang18}, restricting the topology of moduli spaces of two-convex and low entropy hypersurfaces~\cite{BuzanoHaslhoferHersh16, buzanohaslhoferhersh16b, mramorwang2020topological, Mramorunknottedness20}, solving the Schoenflies problem for low entropy three-spheres~\cite{ChodoshChoiMantoulidisSchulze20, bernsteinwang2020}, constructing new minimal surfaces~\cite{haslhoferketover17, song19}, and proving isoperimetric inequalities~\cite{schulze2020optimal, hershkovits2017isoperimetric}. Almost all of these applications use weak solutions to define the flow after the first singular time. The primary obstructions to improving these results is our understanding of weak solutions and potential singularities.

There are two main notions of weak solution for a generic mean curvature flow: Brakke flow and level set flow.\footnote{Flow with surgery is also a type of weak solution, but it is hampered by curvature restrictions on the initial condition. It has not been constructed for a generic initial condition.} One major difference between these two flows is the fact that Brakke flow may suddenly vanish, i.e.\ have mass drop, while level set flow may fatten, i.e.\ attain positive measure. The discrepancy between these two flows is related to important questions about mean curvature flow, like the multiplicity one conjecture and the partial regularity of the flow. In 2006, Metzger-Schulze proved that there is no mass drop for mean convex flow and that equality is achieved in the integral version of the Brakke inequality~\cite{MetzgerSchulze08}. This means that  Brakke flow and level set flow give essentially the same theory in the mean convex setting. In this paper, we will extend Metzger-Schulze's result to the non-mean convex setting, though we use different techniques. This will be done by relating various open problems in mean curvature flow to each other, which will allow us to prove no mass drop and equality in the Brakke inequality for important non-mean convex cases.

Our main result says that an $n$-dimensional Brakke flow in $\bR^{n+k}$, under natural assumptions, has no mass drop if and only if it satisfies the multiplicity one conjecture for $\cH^n$-a.e.\ $x \in \bR^{n+k}$ at each time (Theorem \ref{theorem main result}). Moreover, this is equivalent to the partial regularity condition that the flow is fully smooth $\cH^n$-a.e.\ for each time.  The assumptions in Theorem \ref{theorem main result} hold in particular for any flow constructed by elliptic regularization with smooth closed initial condition (see Theorem \ref{theorem unit density for nonfattening}). As a consequence of Theorem \ref{theorem main result}, we find that a generic level set flow will have a canonical associated Brakke flow with no mass drop, at least until the time at which there is a higher multiplicity planar tangent flow (Corollary \ref{corollary general assumption no mass drop}, Corollary \ref{corollary mean convex neighborhoods of singularities}). Also, level set flows with mean convex neighborhoods of singularities have no mass drop. We then strengthen the no mass drop condition to equality in the integral version of the Brakke flow inequality in some important cases. Under a uniqueness assumption for the level set flow, flows with three-convex blow-up type are Brakke flows with equality in the integral sense (Theorem \ref{theorem main assumption BFE before tau}). Particular examples include flows with generic singularities in $\bR^3$ and flows with mean convex neighborhoods of singularities in $\bR^4$ (Corollary \ref{corollary two convex three convex BFE}). Using Theorem \ref{theorem main result} and Theorem \ref{theorem main assumption BFE before tau}, we rule out quasistatic limit flows of a Brakke flow under a multiplicity one assumption on the Brakke flow. Limit flows at a spacetime point $(x,t)$ have no mass drop if $t$ is before the time there is a higher multiplicity planar limit flow (Corollary \ref{corollary limit flows of BFEs}). Moreover, these limit flows are Brakke flows with equality so long as they are smooth at all but countably many times. Whether a flow is a Brakke flow with equality is closely related to the size of the set of worldlines which originate at singularities. This relationship will be expounded upon in an upcoming paper~\cite{Payne20b}. 

Recent breakthroughs in non-mean convex mean curvature flow will be used to prove some results in this paper. These include Hershkovits-White's result that level set flows with mean convex neighborhoods of singularities are nonfattening~\cite{HershkovitsWhite17}, Choi-Haslhofer-Hershkovits-White's result that level set flows with only spherical singularities and neckpinches are nonfattening~\cite{ChoiHaslhoferHersh18, ChoiHaslhoferHershWhite19}, and Colding-Minicozzi's stratification of the singular set of Brakke flows with generic singularities~\cite{ColMin16, ColMin19}. Some recent breakthroughs also provide examples of flows to which our results apply. These include low entropy flows with only neckpinches and spherical singularities~\cite{bernsteinwang171, BernsteinWang18, bernsteinwang2020, ChodoshChoiMantoulidisSchulze20}, as well as Chodosh-Choi-Mantoulidis-Schulze's generic perturbations of closed surfaces in $\bR^3$, assuming some outstanding conjectures~\cite{ChodoshChoiMantoulidisSchulze20}.

\subsection{Definitions and Conjectures}\hfill\\
\vspace{-.2in}

Most theorem statements will be written from the perspective of both Brakke flow and level set flow. The assumptions on level set flow are more involved, so we define our terms here. See Section \ref{section preliminaries} for background concepts and definitions.

A pair $(M_0, \mu_t)$ satisfies the \textbf{General Assumption} if $M_0$ is a smooth closed embedded hypersurface in $\bR^{n+1}$, the level set flow of $M_0$ is nonfattening, and $\mu_t$ is the unit density Brakke flow associated by elliptic regularization (see Theorem \ref{theorem unit density for nonfattening}).
In more detail:

\noindent \textbf{General Assumption:}

\noindent We say that the pair $(M_0, \mu_t)$ satisfies the General Assumption for $t<T$ if the following holds:
\begin{enumerate}
    \item $M_0 \subset \bR^{n+1}$ is a smooth closed embedded connected hypersurface, $D_0$ is a bounded connected open set with $\dd D_0= M_0$,
    \item The level set flow of $M_0$ is nonfattening for $t<T$, i.e.\ for each $t \in [0, T)$, the level set flow $u: \bR^{n+1} \times \bR_{\geq 0} \to \bR$ with initial conditions $\{u(\cdot, 0) =0\} = M_0$ and $\{u(\cdot, 0) >0\}= D_0$ satisfies
    $$\cH^{n+1}(\{u(\cdot, t)=0\})=0$$
    \item By elliptic regularization~\cite{Ilmanen94}, $\{u(\cdot, t)>0\}$ has finite perimeter for $t \geq 0$, and 
    $$\mu_{t} := \cH^{n} \lfloor  \dd^*\{u(\cdot, t)>0\}$$
    is an integral unit density Brakke flow, where $\dd^*$ denotes the reduced boundary.\hfill\\
\end{enumerate}

We will use the notation that $\mu(\phi)$ denotes the integral of $\phi$ with respect to the Radon measure $\mu$, and $\bf{M}[\mu_{t}]$ denotes the mass of the Radon measure $\mu_{t}$, i.e.\ $\bf{M}[\mu_{t}] = \mu_{t}(1)$. The vector $\vec{H}_{\mu}$ denotes the generalized mean curvature vector of the varifold $V_{\mu}$, when it exists. 

We now state a few conjectures that motivate this work. In this paper, we will relate each of these conjectures to each other in some way. These conjectures are stated for flows with codimension one, but most can be written analogously in higher codimension. Many can also be written with assumptions on just a Brakke flow, instead of on a level set flow. The first conjecture is the multiplicity one conjecture~\cite[pg. 7]{Ilmanen95}~\cite[\#2]{ilmanen2003problems}, perhaps the most famous open problem in mean curvature flow. The multiplicity one conjecture is central to nearly all major open problems in mean curvature flow.

\begin{cj}[Multiplicity One]\label{Conjecture multiplicity one}
Under the General Assumption, each tangent flow of $\mu_t$ is multiplicity one. 
\end{cj}
Tangent flows are themselves Brakke flows and are the parabolic analogue of the approximate tangent spaces of a rectifiable varifold. By multiplicity one, we mean that the tangent flow may be written as $\cH^n \lfloor M_t$ for some flow of sets $M_t$. So, multiplicity one tangent flows are unit density at each time. 

The next conjectures we list, Conjectures \ref{Conjecture A}, \ref{Conjecture B}, and \ref{Conjecture C}, can be found in the appendix of Ilmanen's monograph on elliptic regularization~\cite{Ilmanen94}. They form the heart of this paper.

\begin{cj}[Better Partial Regularity]\label{Conjecture A}
Under the General Assumption, $$\cH^{n}(\sing\mu)=0$$
where $\sing \mu = \spt \mu \setminus \reg \mu$ with $\reg \mu$ the smooth part of $\spt \mu$, the spacetime support. 
\end{cj}

\begin{cj}[No Mass Drop]\label{Conjecture B}
Under the General Assumption, $t \mapsto \bf{M}[\mu_{t}]$ is continuous.
\end{cj}
For a compact initial condition, Conjecture \ref{Conjecture B} is in fact equivalent to the condition that $t \mapsto \mu_t(\phi)$ is continuous for each $\phi \in C^2_c(\bR^{n+1}, \bR_{\geq 0})$ (see Proposition \ref{proposition mass continuous implies mass measure continuous}).
\begin{cj}[Equality in Brakke's Inequality]\label{Conjecture C}
Under the General Assumption, $\mu_{t}$ is a Brakke flow with equality for a.e.\ $t \geq 0$. That is, for a.e.\ $t \geq 0$ and each $\phi \in C^2_c(\bR^{n+1}, \bR_{\geq 0})$,
$$\ov{D}_t \mu_t(\phi) = \int -\phi |\vec{H}_{\mu_t}|^2 + \na \phi \cdot \vec{H}_{\mu_t} \,d\mu_t$$
where $\ov{D}_t$ is the upper derivative and $\vec{H}_{\mu_t}$ is the generalized mean curvature vector for $\mu_t$, when it exists (see Section \ref{section preliminaries Brakke flow}).
\end{cj}

In this paper, we will relate versions of Conjectures \ref{Conjecture multiplicity one}, \ref{Conjecture A}, and \ref{Conjecture B} (see Theorem \ref{theorem main result}). It is generally unknown precisely how Conjectures \ref{Conjecture B} and \ref{Conjecture C} are related, except under countability assumptions on the singular times (see Theorem \ref{theorem main assumption BFE before tau}). 

The next conjecture is implicit in the work of Metzger-Schulze, who proved it in the mean convex case~\cite{MetzgerSchulze08}. It implies Conjectures \ref{Conjecture B} and \ref{Conjecture C}.

\begin{cj}[Integral Brakke Equality]\label{Conjecture D}
Under the General Assumption, for each $s,t \geq 0$ and each $\phi \in C^2_c(\bR^{n+1}, \bR_{\geq 0})$,
$$\mu_t(\phi) - \mu_s(\phi) = \int_{s}^t \int -\phi |\vec{H}_{\mu_{s'}}|^2 + \na \phi \cdot \vec{H}_{\mu_{s'}} \,d\mu_{s'} ds'$$
\end{cj}

Assuming Conjecture \ref{Conjecture D}, $t\mapsto \mu_t(\phi)$ is absolutely continuous. There are examples of non-integral Brakke flows which satisfy Conjecture \ref{Conjecture B} but not Conjecture \ref{Conjecture D}. The idea is to choose a smooth mean curvature flow and assign a spatially constant density $\te(t)$ at time $t$, such that $t \mapsto \te(t)$ is a continuous nonincreasing positive function which is not absolutely continuous. This construction gives a non-integral Brakke flow which has continuous but not absolutely continuous mass, meaning it satisfies Conjecture \ref{Conjecture B} but not Conjecture \ref{Conjecture D}. 

The next conjecture is closely related to a conjecture of Ilmanen's~\cite[p.73]{Ilmanen94} and was stated explicitly in the work of Hershkovits-White~\cite{HershkovitsWhite17}.

\begin{cj}[Uniqueness of Nonfattening Flows]\label{conjecture E}
Suppose $M_0$ is a smooth closed embedded hypersurface. Then,
$$T_{\disc} = T_{\fat}$$
where $T_{\disc}$ is the first time that the level set flow, inner flow, and outer flow are not all the same (see Section \ref{subsection preliminaries of level set flow}) and $T_{\fat}$ is the first fattening time of the level set flow.
\end{cj}
Conjecture \ref{conjecture E} says that non-uniqueness for level set flow occurs precisely when the flow fattens. This holds when the level set flow has mean convex neighborhoods of singularities~\cite{HershkovitsWhite17} or when it has only spherical singularities and neckpinches~\cite{ChoiHaslhoferHersh18, ChoiHaslhoferHershWhite19}. Conjecture \ref{conjecture E} is related to questions about multiplicity. In this paper, we will often assume Conjecture \ref{conjecture E} together with the assumption that no blow-ups of $\mu_t$ are higher multiplicity planes. This multiplicity one assumption is possibly too weak to imply Conjecture \ref{conjecture E} though, since it only provides information on $\mu$, not on the entire level set flow a priori.
\subsection{Statements of Results}\hfill\\
\vspace{-.2in}

Let $\{\mu_t\}_{t \geq 0}$ be an integral $n$-dimensional Brakke flow in $\bR^{n+k}$ with bounded area ratios. The assumption of bounded area ratios is natural and holds in particular for any flow with a smooth initial condition (see Theorem \ref{theorem huisken monotonicity}). 

Define
$$\cT_{\mult}(t) := \{x \in \bR^{n+k}\,|\,\exists \,\text{a planar tangent flow at }(x,t)\, \text{with multiplicity $\geq 2$}\}$$
where the tangent flows are taken with respect to the flow $\mu_t$.

Throughout this paper, we will often assume the flows are unit regular. This is a natural assumption since any enhanced motion constructed via elliptic regularization is unit regular (see Theorem \ref{theorem unit density for nonfattening}). Unit regularity rules out gratuitous vanishing of smooth parts of the flow and is equivalent to assuming that there are no quasistatic multiplicity one planes as tangent flows. 

We define $(\sing^+\hspace{-.04in}\mu)_t$ to be the set of points in the time-$t$ slice of $\spt \mu$ which are not fully smooth (see Section \ref{section singularities mean curvature flow}). For a unit density flow, a point is fully smooth if the support is smooth in a spacetime neighborhood of the point.

The idea of our main result, Theorem \ref{theorem main result}, is that a flow has mass drop only when there is a macroscopic drop in density or when smooth regions suddenly vanish. Drops in density and vanishing of smooth regions are detected by planar tangent flows, either high multiplicity or quasistatic. 

\begin{thm}\label{theorem main result}
Let $T>0$, and let $\{\mu_t\}_{t \geq 0}$ be an integral $n$-dimensional Brakke flow in $\bR^{n+k}$, which is unit density, is unit regular, and has bounded area ratios. Then, the following are equivalent:
\begin{enumerate}
    \item\label{condition 1} For each $t \in (0,T)$, $\cH^n(\cT_{\mult}(t))=0$,
    \item\label{condition 2} For each $t \in (0,T)$, $\cH^n((\sing^+\hspace{-.04in}\mu)_t)=0$,
    \item\label{condition 3} For each $\phi \in C^2_c(\bR^{n+k}, \bR_{\geq 0})$, $t \mapsto \mu_t(\phi)$ is continuous for $t \in (0,T)$.
\end{enumerate}
\end{thm}

Note that if $\mu_t$ is smooth with bounded curvature until it vanishes at time $T$, then $\cH^n((\sing^+ \hspace{-.04in}\mu)_T)>0$ and $\mu_t$ has mass drop at time $T$. So, Theorem \ref{theorem main result} does not apply through time $T$ for a smooth flow which suddenly vanishes at time $T$.

By White's stratification theorem ~\cite[Theorem 9]{Wh97} (cf. \cite{CheegerHaslhoferNaber13}), for each time $t>0$, the set of points $x \in \bR^{n+k}$ such that there are only nonplanar tangent flows at $(x,t)$ has $\cH^n$ measure zero (see Theorem \ref{theorem white stratification}). Thus, condition (\ref{condition 1}) of Theorem \ref{theorem main result} is equivalent to assuming that for each $t \in (0,T)$ and for $\cH^n$-a.e.\ $x \in \bR^{n+k}$, each tangent flow at $(x,t)$ is multiplicity one. In other words, condition (\ref{condition 1}) is equivalent to assuming that the multiplicity one conjecture holds $\cH^n$-a.e.\ in each time-$t$ slice. Condition (\ref{condition 1}) may be rightfully called the $\bf{\cH^n}$\textbf{-a.e.\ multiplicity one conjecture}.

We also prove a generalization of Theorem \ref{theorem main result} without the assumption of unit regularity (see Theorem \ref{theorem main result without unit regularity}). It turns out that conditions (\ref{condition 2}) and (\ref{condition 3}) of Theorem \ref{theorem main result} hold without a unit regularity assumption, as unit regularity is effectively built into these conditions already. Without the unit regularity assumption, we must modify condition (\ref{condition 1}) to rule out quasistatic multiplicity one planes. Also, a weaker version of Theorem \ref{theorem main result} holds without a unit density assumption. The unit density assumption is effectively built into conditions (\ref{condition 1}) and (\ref{condition 2}) by definition (see Theorem \ref{theorem main result without unit regularity} and Remark \ref{remark unit density unnecessary}). 

\begin{rmk}
If $(M_0, \mu_t)$ satisfies the General Assumption for $t<T$, then condition (\ref{condition 2}) of Theorem \ref{theorem main result} is satisfied for a.e. $t \in (0,T)$ by Ilmanen's regularity result \cite[Theorem 12.9, 12.11]{Ilmanen94}. The assumption that condition (\ref{condition 2}) holds for every $t \in (0,T)$, as opposed to almost every $t$, is crucial for proving no mass drop.
\end{rmk}

Define
$$T_{\mult}:= \inf\{t >0 \,|\,\exists  \text{ a planar tangent flow at }(x,t)\text{ with multiplicity}\geq 2\}$$
The first corollary, Corollary \ref{corollary general assumption no mass drop}, is an application of Theorem \ref{theorem main result}. It says that nonfattening flows which satisfy Conjecture \ref{conjecture E} and the multiplicity one conjecture satisfy the no mass drop conjecture.

\begin{cor}\label{corollary general assumption no mass drop}
Suppose that $(M_0, \mu_t)$ satisfies the General Assumption for 
$$t< T= \min(T_{\disc}, T_{\mult})$$
Then, for each $\phi \in C^2_c(\bR^{n+1}, \bR_{\geq 0})$, $t \mapsto \mu_t(\phi)$ is continuous for $t \in [0,T)$.
\end{cor}
Corollary \ref{corollary general assumption no mass drop} may be applied to a generic flow. A generic flow is nonfattening and will satisfy Conjecture \ref{conjecture E} by Proposition \ref{proposition general assumption is generic}. So, we find that a generic flow will satisfy the no mass drop conjecture until the time there is a planar tangent flow of multiplicity $\geq 2$. Corollary \ref{corollary general assumption no mass drop} also applies to level set flows with mean convex neighborhoods of singularities. Flows with mean convex neighborhoods of singularities have no higher multiplicity planar tangent flows and satisfy Conjecture \ref{conjecture E} by White's regularity theory for mean convex flows~\cite{Wh00} and work of Hershkovits-White~\cite{HershkovitsWhite17}. As an application of Corollary \ref{corollary general assumption no mass drop}, we find that flows with mean convex neighborhoods of singularities satisfy the no mass drop conjecture, Conjecture \ref{Conjecture B}, for all time. 



\begin{cor}\label{corollary mean convex neighborhoods of singularities}
Let $M_0$ be a smooth closed embedded hypersurface such that one of the following assumptions holds:
\begin{itemize}
    \item The level set flow of $M_0$ has mean convex neighborhoods\footnote{In the sense of \cite{ChoiHaslhoferHershWhite19}. See the definition in Section \ref{subsection preliminaries of level set flow}.} of singularities and $T = \infty$,
    \item $M_0$ is a generic\footnote{Let N be a smooth closed embedded hypersurface, and let $f: N \times [-1,1] \to \bR^{n+1}$ be a smooth one-parameter family of graphs $f(\cdot, s)$ over $N$. There is a full measure set $J \subseteq [-1,1]$ such that for $s \in J$, $M_0 = f(N,s)$ is nonfattening and $\mu_t$ satisfies the conclusion of Corollary \ref{corollary mean convex neighborhoods of singularities} for $t \in [0, T_{\mult})$.} closed hypersurface and $T = T_{\mult}$.
\end{itemize}
Let $\mu_t$ be the unit density Brakke flow associated to $M_0$ via elliptic regularization.

Then, for each $\phi \in C^2_c(\bR^{n+1}, \bR_{\geq 0})$, $t \mapsto \mu_t(\phi)$ is continuous for $t \in [0,T)$.
\end{cor}

Metzger-Schulze proved that mean convex flows achieve equality in the Brakke inequality, Conjecture \ref{Conjecture D}, which is a stronger condition than the no mass drop conjecture, Conjecture \ref{Conjecture B}~\cite{MetzgerSchulze08}. Their proof uses the assumption that the flow is globally mean convex, unlike our assumption of local mean convexity around singularities in Corollary \ref{corollary mean convex neighborhoods of singularities}. Metzger-Schulze's argument is difficult to localize to mean convex neighborhoods of singularities, as it would require showing that the approximators from elliptic regularization are mean convex near each singularity. This is certainly difficult in general, but it may be possible for flows with only spherical singularities and neckpinches using an argument of Schulze-Sesum \cite[Prop. 2.3]{sesumschulze2020stability} based on the resolution of the mean convex neighborhood conjecture in this case~\cite{ChoiHaslhoferHersh18, ChoiHaslhoferHershWhite19}. With that said, this is not the approach that we take in this paper. With an approach different from that of Metzger-Schulze, we ultimately conclude that Conjecture \ref{Conjecture D} holds for a general class of flows which includes as a special case the flows with only spherical singularities and neckpinches (see Corollary \ref{corollary two convex three convex BFE}). 

We say that an $n$-dimensional Brakke flow $\{\mu_t\}_{t \geq 0}$ in $\bR^{n+k}$ has a \textit{countable set of singular times} for $t<T$ if there is a cocountable set $\cI \subseteq [0, T)$ such that for each $t_0 \in \cI$, $\Te(x,t_0)\leq 1$ for each $x \in \bR^{n+k}$, where $\Te(x,t_0)$ is the Gaussian density at the spacetime point $(x,t_0)$ (see Section \ref{section singularities mean curvature flow}). 

We say that a level set flow $u: \bR^{n+1}\times \bR_{\geq 0} \to \bR$ has an \textit{open cocountable set of regular times} for $t<T$ if there is a relatively open cocountable set $\cI \subseteq [0, T)$ such that for each $t_0 \in \cI$, $\{u(\cdot, t_0)=0\}$ is a smooth embedded, possibly disconnected, hypersurface.

The next theorem, Theorem \ref{theorem main assumption BFE before tau}, says that flows satisfying one of the aforementioned countability assumptions on the singular set satisfy Conjectures \ref{Conjecture B}, \ref{Conjecture C}, and \ref{Conjecture D} until there is a planar tangent flow which is not static and multiplicity one. In fact, it implies that $t \mapsto \mu_t(\phi)$ is absolutely continuous until time $T$, improving the result of Corollary \ref{corollary general assumption no mass drop}.

The idea of Theorem \ref{theorem main assumption BFE before tau} is that the countability assumption on the singular times allows one to upgrade mass continuity from Theorem \ref{theorem main result} to absolute continuity. By proving that Conjecture \ref{Conjecture C} is true under these assumptions and then integrating the absolutely continuous mass, we find that the flows satisfy the integral Brakke equality, Conjecture \ref{Conjecture D}.

\begin{thm}\label{theorem main assumption BFE before tau}
Suppose that one of the following assumptions holds:
\begin{itemize}
  \item $(M_0, \mu_t)$ satisfies the General Assumption for $t< T= \min(T_{\disc}, T_{\mult})$, the level set flow of $M_0$ has an open cocountable set of regular times for $t<T$, and $k=1$,
  \item $\mu_t$ is an integral $n$-dimensional Brakke flow in $\bR^{n+k}$ which is unit regular, has bounded area ratios, and has a countable set of singular times for $t<T= T_{\mult}$.
\end{itemize}
Then, for $\phi \in C^2_c(\bR^{n+k}, \bR_{\geq 0})$ and $0 < s,t \leq T$,
    $$
    \mu_{t}(\phi) - \mu_{s}(\phi) = \int_{s}^t \int -\phi |\vec{H}_{\mu_{s'}}|^2 + \na \phi \cdot \vec{H}_{\mu_{s'}} \,d\mu_{s'}\,ds'
    $$
\end{thm}

The next corollary is an application of Theorem \ref{theorem main assumption BFE before tau}. It uses Colding-Minicozzi's stratification theory for flows with generic singularities to find examples of flows with countable sets of singular times(see Proposition \ref{proposition examples of main assumption})~\cite{ColMin16, ColMin19} . It also uses the recent breakthrough that level set flows with only neckpinches and spherical singularities are nonfattening and satisfy $T_{\disc} = \infty$~\cite{ChoiHaslhoferHersh18, ChoiHaslhoferHershWhite19, HershkovitsWhite17}.

\begin{cor}\label{corollary two convex three convex BFE}
Let $M_0$ be a smooth closed embedded hypersurface such that one of the following assumptions holds:
\begin{itemize}
    \item The level set flow of $M_0$ has only spherical singularities and neckpinches,\footnote{In the sense of Choi-Haslhofer-Hershkovits-White~\cite{ChoiHaslhoferHershWhite19}. See Section \ref{subsection preliminaries of level set flow}.} and $T = \infty$,
    \item The level set flow of $M_0$ has mean convex neighborhoods of singularities in $\bR^{n+1}$, $n+1 \leq 4$, and $T=\infty$,
    \item The level set flow of $M_0$ has three-convex blow-up type\footnote{In the sense that all blow-ups are the blow-ups of three-convex flow. See Section \ref{subsection preliminaries of level set flow}.} and $T = T_{\disc}$.
\end{itemize}
Let $\mu_t$ be the unit density Brakke flow associated to $M_0$ via elliptic regularization.

Then, for $\phi \in C^2_c(\bR^{n+1}, \bR_{\geq 0})$ and $0 \leq s,t \leq T$,
    $$\mu_{t}(\phi) - \mu_{s}(\phi) = \int_{s}^t \int -\phi |\vec{H}_{\mu_{s'}}|^2 + \na \phi \cdot \vec{H}_{\mu_{s'}} \,d\mu_{s'}\,ds'$$
\end{cor}

The conclusion of Corollary \ref{corollary two convex three convex BFE} also holds until time $T = T_{\mult}$ for a unit density, unit regular Brakke flow with bounded area ratios and three-convex blow-up type. One interpretation of Corollary \ref{corollary two convex three convex BFE} is that Conjectures \ref{Conjecture B}, \ref{Conjecture C}, and \ref{Conjecture D} are true for generic flows with generic singularities in $\bR^{n+1}$, $n+1 \leq 4$.

Recently, Chodosh-Choi-Mantoulidis-Schulze proved that generic perturbations of a closed embedded hypersurface in $\bR^3$ have only spherical singularities and neckpinches so long as there are no higher multiplicity tangent flows and the no cylinder conjecture holds~\cite{ChodoshChoiMantoulidisSchulze20}. These flows satisfy the conclusion of Corollary \ref{corollary two convex three convex BFE}, using the formulation of the corollary for Brakke flows. Likewise, a low entropy assumption on the initial condition will ensure that only spherical singularities and neckpinches will occur for the flow, and such flows would also satisfy the conclusion of Corollary \ref{corollary two convex three convex BFE}.

We may apply Theorem \ref{theorem main result} and Theorem \ref{theorem main assumption BFE before tau} to improve the structure of limit flows in some important cases. Limit flows are a generalization of tangent flows, and unlike tangent flows, relatively little is known about them. For the definition of limit flows, see Section \ref{section singularities mean curvature flow}. The structure of limit flows, as opposed to just tangent flows, has been crucial for geometric and topological applications of mean curvature flow (e.g., see \cite{ChoiHaslhoferHersh18, ChoiHaslhoferHershWhite19}).

For an integral Brakke flow $\mu_t$ with bounded area ratios, we define 
$$T_{\mult}^*:= \inf\{t >0 \,|\,\exists  \text{ a planar limit flow at }(x,t)\text{ with multiplicity}\geq 2\}$$

\begin{cor}\label{corollary limit flows of BFEs}
Suppose that one of the following assumptions holds:
\begin{itemize}
    \item $(M_0, \mu_t)$ satisfies the General Assumption for $t<T= \min(T_{\disc}, T^*_{\mult})$, and $k=1$,
    \item $\mu_t$ is an integral $n$-dimensional Brakke flow in $\bR^{n+k}$ which is unit regular, has bounded area ratios, and $T= T^*_{\mult}$.
\end{itemize}
Let $\mu'_t$ be a limit flow of $\mu_t$ at a spacetime point $(x_0,t_0)$ with $t_0\in (0,T)$. 

Then, for $\phi \in C^2_c(\bR^{n+k}, \bR_{\geq 0})$, $t \mapsto \mu'_t(\phi)$ is continuous for $-\infty < t < \infty$. 

Moreover, if $\mu'_t$ has a countable set of singular times, then for $\phi \in C^2_c(\bR^{n+k}, \bR_{\geq 0})$ and $-\infty < s,t < \infty$,
    $$\mu'_{t}(\phi) - \mu'_{s}(\phi) = \int_{s}^t \int -\phi |\vec{H}_{\mu'_{s'}}|^2 + \na \phi \cdot \vec{H}_{\mu'_{s'}} \,d\mu'_{s'}\,ds'$$
\end{cor}
Limit flows of a unit regular Brakke flow are unit regular by~\cite[Section 7]{White05}. Also, restrictions on limit flows of a Brakke flow will pass to restrictions on tangent flows of a limit flow (see Lemma \ref{lemma limit flow of limit flow is limit flow}). Combining these ideas, Corollary \ref{corollary limit flows of BFEs} formalizes the idea that unit regularity and multiplicity assumptions on $\mu_t$ rule out all microscopic drops in mass detected by limit flows. By Lemma \ref{lemma limit flow of limit flow is limit flow}, the assumption that $t<T_{\mult}^*$ is tantamount to the general multiplicity one conjecture holding until time $T_{\mult}^*$.

White proved that no limit flow of a mean convex flow is a higher multiplicity plane, and limit flows are smooth until they vanish~\cite{Wh03, wh11} (cf.~\cite{HaslhoferKleiner171}). Each limit flow of a mean convex flow satisfies the assumptions of Corollary \ref{corollary limit flows of BFEs} with $T=\infty$. Corollary \ref{corollary limit flows of BFEs} can be seen as a generalization of the fact that limit flows of mean convex flows achieve equality in the integral version of the Brakke inequality.
\vspace{-.05in}

\subsection{Method of Proof}\label{section method of proof}\hfill\\
\vspace{-.2in}

The idea of Theorem \ref{theorem main result} is that a Brakke flow can only have mass drop when there is a drop in density on a large set or when some large region of the support vanishes. 

Intuitively, drops in density are detected by higher multiplicity planar tangent flows. The flow may converge with multiplicity at a singular time and then immediately drop to unit density or disappear, resulting in a drop in mass. The assumption on $\cT_{\mult}(t)$ is designed to rule out all higher multiplicity planar tangent flows. On the other hand, sudden vanishing of the support is detected by quasistatic planar tangent flows. The unit regularity assumption rules out quasistatic multiplicity one planar tangent flows. 

We prove that the flow $\mu_t$ has continuous mass if planar tangent flows which are quasistatic or higher multiplicity are ruled out $\cH^n$-a.e. Without these possible tangent flows, the singular set at each time-slice of the spacetime support has $\cH^n$ measure zero by White's stratification theorem (Theorem \ref{theorem white stratification}). With this regularity, we argue that the flow passes through each time without mass drop. There are static multiplicity one planar tangent flows at $\cH^n$-a.e.\ point of the spacetime support at each time, which means the flow will not drop in mass around most points. Bounds on area ratios imply that mass does not accumulate around the small singular set. Applying a partition of unity around the $\cH^n$ measure zero singular set, we find that there is no mass drop (Theorem \ref{theorem equivalence of mass continuity, codim 2}). 

We then prove the converse statement that continuity of mass implies the singular set is small, i.e. that $\cH^n$-a.e.\ point of the support at each time is fully smooth. This is based on an application of the Brakke regularity theorem, which gives a particularly strong regularity condition assuming mass continuity~\cite[Theorem 6.12]{Brakke1978}. Brakke's argument in \cite[Theorem 6.12]{Brakke1978} shows that at each time, the set of points which have microscopic drops in area ratio has $\cH^n$ measure zero (see Lemma \ref{lemma lahiri's lemma},~\cite[Lemma 9.5]{lahiri2014regularity}). Using Brakke's clearing out lemma, the set of points in the spacetime support of $\mu$ at time $t$ differs from $\spt \mu_t$ by a set of $\cH^n$ measure zero (Lemma \ref{lemma lahiri's support lemma})~\cite[Theorem 9.7]{lahiri2014regularity}. We find that a mass continuous flow is fully smooth at $\cH^n$-a.e.\ point of the spacetime support at each time. This is the desired partial regularity condition, and it is equivalent to the $\cH^n$-a.e. multiplicity one conjecture by White's stratification theorem. This concludes the proof of Theorem \ref{theorem main result}.

Corollary \ref{corollary general assumption no mass drop} is an application of Theorem \ref{theorem main result} to level set flows. The assumption that $t < T_{\mult}$ rules out all higher multiplicity planes as tangent flows. The flows constructed by elliptic regularization, as limits of translators, are unit regular, but this may not be the case for the unit density $\mu_t$. The flow $\mu_t$ merely \textquotedblleft sits underneath" some limit of translators (an \textquotedblleft enhanced motion" in the parlance of ~\cite{Ilmanen94}). Under the assumption that $t < T_{\disc}$, it follows that $\mu_t$ is unit regular, which rules out quasistatic multiplicity one planar tangent flows. Corollary \ref{corollary general assumption no mass drop} then follows from Theorem \ref{theorem main result}. The $t < T_{\disc}$ assumption also says that $\spt \mu$ will coincide with the level set flow up until time $T_{\disc}$. This ensures that the blow-ups of the level set flow coincide with the blow-ups of $\mu_t$. Moreover, generic flows are nonfattening and satisfy $T_{\disc} = T_{\fat} = \infty$. Corollary \ref{corollary mean convex neighborhoods of singularities} follows from these facts and from an application of a result of Hershkovits-White~\cite{HershkovitsWhite17}.

With an assumption on the countability of singular times, we upgrade Theorem \ref{theorem main result} and Corollary \ref{corollary general assumption no mass drop} to a proof of Theorem \ref{theorem main assumption BFE before tau}. The two different assumptions on the Brakke flow and the level set flow in Theorem \ref{theorem main assumption BFE before tau} are effectively the same. These assumptions will imply that all but countably many time-slices of the flow are fully smooth and coincide with a smooth flow. It follows that the flow achieves equality in the derivative version of the Brakke inequality, Conjecture \ref{Conjecture C}, for all but countably many times. We then use the Banach-Zaretskii theorem to prove that $t \mapsto \mu_t(\phi)$ is absolutely continuous. The fact that Conjecture \ref{Conjecture C} holds for all but countably many times is crucial to the application of the Banach-Zaretskii theorem. The absolute continuity of mass is used to show that the Brakke flow achieves equality in the integral Brakke inequality, concluding Theorem \ref{theorem main assumption BFE before tau}. Corollary \ref{corollary two convex three convex BFE} follows from an application of a result of Colding-Minicozzi~\cite{ColMin16, ColMin19} combined with work of Choi-Haslhofer-Hershkovits-White~\cite{ChoiHaslhoferHersh18, ChoiHaslhoferHershWhite19}. Corollary \ref{corollary limit flows of BFEs} is an application of Corollary \ref{corollary general assumption no mass drop} and Theorem \ref{theorem main assumption BFE before tau}, applied directly to limit flows. This requires the folklore result that restrictions on limit flows pass to restrictions on tangent flows of a limit flow.
\vspace{-.05in}

\subsection{Outline of the Paper}\hfill\\
\vspace{-.2in}

In Section \ref{section preliminaries}, we frontload with preliminaries and relevant notions of singularities and singular sets. We include a couple lemmas and theorems which will be used throughout this paper.

In Section \ref{section characterization of the main assumption}, we prove basic facts about flows satisfying the General Assumption and flows which have a countable set of singular times. We show that the $t < T_{\disc}$ assumption implies that there are no quasistatic multiplicity one planar tangent flows and the support of $\mu$ coincides with the level set flow. We relate blowups for level set flow to blowups of Brakke flow. We provide examples of flows with countable sets of singular times and show that Conjecture \ref{Conjecture C} holds for all but countably many times for such flows.

In Section \ref{section partial regularity and mass continuity}, we prove that there is no mass drop for flows with small singular sets (Theorem \ref{theorem equivalence of mass continuity, codim 2}), and conversely, we show that no mass drop implies that there is a small singular set (Theorem \ref{theorem mass continuity implies mult one}). Using stratification, we prove Theorem \ref{theorem main result}. Using statements from Section \ref{section characterization of the main assumption}, we prove Corollary \ref{corollary general assumption no mass drop} and Corollary \ref{corollary mean convex neighborhoods of singularities}. We then upgrade the mass continuity results with the assumption that there is a countable set of singular times, proving Theorem \ref{theorem main assumption BFE before tau} and then Corollary \ref{corollary two convex three convex BFE}. Corollary \ref{corollary limit flows of BFEs} follows as an application of Theorems \ref{theorem main result} and \ref{theorem main assumption BFE before tau} and a result on limit flows.

\subsection{Acknowledgments}\hfill\\
\vspace{-.2in}

The author would like to thank his advisor, Bruce Kleiner, for multiple helpful conversations and for useful methodological advice which helped initiate this paper. The author would also like to thank Felix Schulze, Robert Haslhofer, Alex Mramor, Or Hershkovits, and Tobias Colding for extensive helpful comments and discussions about this paper. This work was partially supported by NSF grant DMS-1711556.

\section{Preliminaries}\label{section preliminaries}

Let $k$ be a positive integer. The measure $\mu$ will denote a Radon measure on $\bR^{n+k}$, and we let $\mu(\phi) := \int \phi\, d\mu$
when this integral exists. We will assume $k=1$ when working with level set flows and $k\in \bZ_{>0}$ otherwise. The family $\{\mu_t\}_{t \geq 0}$ will be a one-parameter family of Radon measures on $\bR^{n+k}$. We will sometimes abbreviate this family by $\mu_t$ when the meaning is clear. We will sometimes consider time intervals other than $t \geq 0$, but we will make note of that when necessary.

We define $B_r(x)$ to be a ball around $x$ of radius $r$. We will assume that balls are in $\bR^{n+k}$. When necessary, we will write $B^{m}_r(x)$ to denote that the ball is in $\bR^{m}$.

For spacetime $\bR^{n+k} \times \bR_{\geq 0}$, let $\tau$ be the time function. So, $\tau: \bR^{n+k} \times \bR_{\geq 0} \to \bR_{\geq 0}$ is defined by $\tau(x,t) = t$.

\subsection{Brakke Flow}\label{section preliminaries Brakke flow}\hfill\\
\vspace{-.2in}

Introduced by Brakke in the 1970s~\cite{Brakke1978}, Brakke flow is a flow of varifolds. It is defined with a variational inequality in order to have a compactness theory in general. The inequality allows for the possibility of a discontinuous mass function over time, i.e.\ mass drop. This is unlike smooth mean curvature flow which has no mass drop and achieves equality in the variational inequality.

We will adhere to much of the notation and definitions in Ilmanen's monograph~\cite{Ilmanen94}. See \cite{simon2014introduction} for background on geometric measure theory. A Radon measure $\mu$ is $n$-rectifiable if $\mu(B) = \int_{B} \te\,d\cH^n$ for $B \subseteq \bR^{n+k}$, where $\te = \te(\mu, \cdot): \bR^{n+k} \to \bR_{\geq 0}$ is a locally $\cH^n$-integrable function and $\{\te >0\}$ is $\cH^n$-measurable and a countably rectifiable set. We say that $\mu$ is integer $n$-rectifiable if $\te$ takes values in the nonnegative integers. If $\mu$ is $n$-rectifiable, we let $V_{\mu}$ be the associated $n$-rectifiable varifold. We denote the first variation of $V_{\mu}$ by $\de V_{\mu}$ (see \cite{allard1972}), and let $|\de V_{\mu}|$ be the total variation of $\de V_{\mu}$. If $|\de V_{\mu}|$ is a Radon measure, then there exists $\nu: \bR^{n+k} \to \bR^{n+k}$ such that $\nu$ is $|\de V_{\mu}|$-measurable, $|\nu(x)|=1$ for $|\de V_{\mu}|$-a.e.\ $x$, and for $X \in C^1_c(\bR^{n+k}, \bR^{n+k})$,
$$\de V_{\mu}(X) = \int \langle \vec{H}_{\mu}, X \rangle \,d\mu -  \int \langle \nu, X\rangle \, d|\de V_{\mu}|_{\text{sing}}$$
where $\vec{H}_{\mu} := H_{\mu}\nu$ with $H_{\mu}:= \frac{d|\de V_{\mu}|}{d\mu}$ and $|\de V_{\mu}|_{\text{sing}} := |\de V_{\mu}|\lfloor \{\frac{d|\de V_{\mu}|}{d\mu}=\infty\}$. We call $\vec{H}_{\mu}$ the generalized mean curvature. Let $T_x \mu$ be the approximate tangent plane of a rectifiable Radon measure $\mu$. In this section, we denote by $(T_x \mu)^{\perp} \cdot \vec{H}_{\mu}(x)$ the projection of the vector $\vec{H}_{\mu}(x)$ onto $(T_x \mu)^{\perp}$ (see Remark \ref{remark Brakke orthogonality}).
\begin{defn}\label{definition Brakke functional}
Let $\mu$ be a Radon measure on $\bR^{n+k}$, and let $\phi \in C^2_c(\bR^{n+k}, \bR_{\geq 0})$. 

\noindent Under the following assumptions:
\begin{enumerate}
    \item $\mu \lfloor \{\phi >0\}$ is a $n$-rectifiable Radon measure,
    \item  $|\de V_{\mu}| \lfloor \{\phi >0\}$ is a Radon measure,
    \item  $|\de V_{\mu}| \lfloor \{\phi >0\}_{\text{sing}}=0$,
    \item 
    $\int \phi H_{\mu}^2 \,d\mu < \infty$,
\end{enumerate}
we define
\begin{equation*}\label{equation brakke variation}\cB(\mu, \phi) := \int -\phi(x) H_{\mu}^2(x) + \na \phi(x) \cdot (T_x \mu)^{\perp} \cdot \vec{H}_{\mu}(x) \,d\mu(x) \end{equation*}
If any of the above assumptions fails, define $\cB(\mu, \phi) := -\infty$.

\end{defn}

Recall that for a function $f: \bR \to \bR$, the upper derivative is defined as
\begin{equation*}
    \ov{D}_{t_0}f := \limsup_{t \to t_0} \frac{f(t) - f(t_0)}{t-t_0}
\end{equation*}

\begin{defn}
The family $\{\mu_t\}_{t \geq 0}$ is a \textbf{Brakke flow} if for each $\phi \in C^2_c(\bR^{n+k}, \bR_{\geq 0})$ and each $t \geq 0$, 
\begin{equation*}
    \ov{D}_{t} \mu_t (\phi) \leq \cB(\mu_{t}, \phi)
\end{equation*}
Alternatively, $\mu_t$ is a Brakke flow if for each $\phi \in C^2_c(\bR^{n+k}, \bR_{\geq 0})$ and each $s,t \geq 0$,
\begin{equation*}
    \mu_t(\phi) - \mu_s(\phi) \leq \int_{s}^t \cB(\mu_{s'}, \phi)\,ds'
\end{equation*}
These two definitions are equivalent for integral Brakke flows (see Definition \ref{definition BF})~\cite{AL17}.
\end{defn}

We will say that a Brakke flow is \textquotedblleft $n$-dimensional" if it is a flow of integer $n$-rectifiable varifolds.

\begin{rmk}\label{remark Brakke orthogonality}
Brakke proved that for integer rectifiable varifolds satisfying (1) and (2) in the above definition, $\vec{H}$ is orthogonal to $T_x \mu$, $\mu$-a.e.\ ~\cite[Section 4]{Brakke1978}. If $\cB(\mu, \phi)>-\infty$ for an integer rectifiable $\mu$, then for $\mu$-a.e.\ x,
$$\na \phi(x) \cdot (T_x \mu)^{\perp} \cdot \vec{H}_{\mu}(x) = \na \phi(x) \cdot \vec{H}_{\mu}(x)$$ 
\end{rmk}
We will apply Remark \ref{remark Brakke orthogonality} implicitly, using the fact that for a Brakke flow, $\cB(\mu_t, \phi)>-\infty$ for almost every $t\geq 0$~\cite[7.2(iv)]{Ilmanen94}. 

\begin{defn}\label{definition BF}
Let $\{\mu_t\}_{t \geq 0}$ be an $n$-dimensional Brakke flow. 

\noindent The flow $\mu_t$ is \textbf{integral} if for a.e.\ $t \geq 0$, $\mu_t$ is an integer $n$-rectifiable varifold. 

\noindent The flow $\mu_t$ has \textbf{unit density} if for each $t \geq 0$, $\Te^n(\mu_{t}, x) = 1$ for $\mu_{t}$-a.e.\ $x$, where the limit $\Te^n(\mu_t, x) = \lim_{r \to 0} \frac{\mu_t(B_r(x))}{\omega_n r^n}$ is well-defined $\mu_t$-a.e.\ ~\cite{simon2014introduction}.

\noindent The flow $\mu_t$ has \textbf{locally bounded mass} if for each $K \subset \subset \bR^{n+k}$, $\sup_{t \geq 0}\, \mu_t(K) < \infty$.
\end{defn}

The next lemma will be used throughout this paper.

\begin{lem}[\cite{Brakke1978}~{\cite[7.2]{Ilmanen94}}]\label{lemma properties of BF}
Let $\{\mu_t\}_{t\geq 0}$ be a Brakke flow with locally bounded mass. Then, for each $\phi \in C^2_c(\bR^{n+k}, \bR_{\geq 0})$, the following properties hold: 
\begin{enumerate}
    \item There exists $C(\phi)$ such that the following quantity is nonincreasing in time:
    $$\mu_t(\phi) - C(\phi)t$$
    \item For each $t>0$,
    $$\lim_{s \nearrow t} \mu_s(\phi) \geq \mu_t(\phi) \geq \lim_{s\searrow t} \mu_s(\phi)$$
    and each limit exists. If $t=0$, $\mu_0(\phi) \geq \lim_{s \searrow 0}\mu_s(\phi)$ and this limit exists.
\end{enumerate}
\end{lem}

\subsection{Level Set Flow}\label{subsection preliminaries of level set flow}\hfill\\
\vspace{-.2in}

Level set flow is a flow of closed sets which coincides with smooth mean curvature flow when the sets are smooth hypersurfaces. It was first rigorously formulated by Evans-Spruck~\cite{EvansSpruck91} and Chen-Giga-Goto~\cite{ChenGigaGoto91}. Level set flow is given by a unique time-varying function whose level sets satisfy the mean curvature flow equation in a weak sense. 

Let $\Ga \subset \subset \bR^{n+1}$, and let $g: \bR^{n+1} \to \bR$ be a continuous function such that $\Ga = \{g = 0\}$ and all but at most one level set of $g$ is compact. Then, there exists a unique weak solution $u: \bR^{n+1} \times \bR_{\geq 0} \to \bR$ solving
\[
\begin{cases}
    \,\,u_t = \sum_{i, j=1}^{n+1} \big(\de_{ij} - \frac{u_{x_i} u_{x_j}}{|Du|^2}\big)u_{x_i x_j}& \text{on } \bR^{n+1}\times (0, \infty)\\
    \,\,\, \qquad \qquad u = g,              & \text{on } \bR^{n+1}\times \{0\}
\end{cases}
\]
The level set flow of $\Ga$ is denoted by $F_t(\Ga) := \{u(\cdot, t)=0\}$ and this is independent of the choice of $g$.

Let $M_0$ be a smooth closed embedded connected hypersurface and let $D_0$ be a bounded connected open set with $\dd D_0 = M_0$. Let $D'_0$ be the complement of $D_0$. There exist three possible flows of sets we can associate to $M_0$: the level set flow, the outer flow, and the inner flow. 

The level set flow of $M_0$ is $F_t(M_0)$. The outer and inner flows are constructed from the level set flows of $\ov{D_0}$ and $D'_0$, respectively. Define
$$\cU:= \{(x,t) \in \bR^{n+1} \times [0, \infty)\,|\,x \in F_t(\ov{D_0})\}$$
$$\cU':= \{(x,t) \in \bR^{n+1} \times [0, \infty)\,|\,x \in F_t(D'_0)\}$$
The \textbf{outer flow} of $M_0$ is defined by 
$$t \mapsto F_t^{\text{out}}(M_0) := \{x \in \bR^{n+1}\,|\, (x,t) \in \dd \cU\}$$
The \textbf{inner flow} of $M_0$ is defined by
$$t \mapsto F_t^{\text{in}}(M_0) := \{x \in \bR^{n+1} \,|\, x \in \dd \cU'\}$$
Here, $\dd \cU$ and $\dd \cU'$ denote the relative boundaries of $\cU$ and $\cU'$ in spacetime $\bR^{n+1} \times \bR_{\geq 0}$.

One important feature of level set flow is that it may fatten. This means that the level sets may attain positive $\cH^{n+1}$-measure, even if the initial condition is smooth~\cite{angenentilmanenchopp95, white02}. If $F_t(M_0)$ fattens, then $F_t^{\text{out}}(M_0)$ and $F_t^{\text{in}}(M_0)$ are distinct. Fattening implies non-uniqueness of level set flow, since it forces the inner and outer flows to be distinct. It is unknown in general if the inner and outer flows can be distinct before the flow fattens. This is the content of Conjecture \ref{conjecture E}. Following the terminology of \cite{HershkovitsWhite17},
$$T_{\disc} := \inf\big\{t>0\,|\,F_t(M_0), F_t^{\text{out}}(M_0), \text{and } F_t^{\text{in}}(M_0) \text{ are not all equal}\big\}$$
$$T_{\fat} := \inf\big\{t>0\,|\,F_t(M_0) \text{ has non-empty interior}\big\}$$
Conjecture \ref{conjecture E} says that $T_{\disc} = T_{\fat}$ for any level set flow. This conjecture implies that if the level set flow, outer flow, and inner flow are distinct at some time, then the level set flow is fattening at that time.

We say that a spacetime point $(x,t)$ with $x \in F_t^{\text{out}}(M_0)$ is \textbf{backwardly regular} if there exists a ball $B_{\eps}(x)$ such that $s \mapsto F_s^{\text{out}}(M_0) \cap B_{\eps}(x)$ is a smooth embedded mean curvature flow for $s \in [t-\eps^2, t]$. If a point $(x,t)$ is not backwardly regular, then it is called \textbf{backwardly singular}. 

For a spacetime point $X = (x_0,t_0)$, we define $\cU_{X,\la}$ by
$$\cU_{X,\la}:= \{(x,t) \,|\,x = \la(y - x_0), t = \la^2(s-t_0) \text{ for some }(y,s) \in \cU\}$$

The level set flow of $M_0$ has only \textbf{spherical singularities} and \textbf{neckpinches} if for each backwardly singular spacetime point $X$, $\tau(X)>0$, $\cU_{X,\la}$ converges smoothly with multiplicity one as $\la \to \infty$ to a shrinking $B^{n+1}$ or $B^n \times \bR$. A shrinking $B^{n+1-k} \times \bR^k$ is the flow $\sqrt{-2(n-k)t}\,B^{n+1-k} \times \bR^k$. Smooth convergence with multiplicity one  means that in each compact set, $\cU_{X, \la}$ converges as sets in the Hausdorff sense and $\dd \cU_{X,\la}$ converges smoothly as a graph over the boundary of the limit. We take a full limit in this definition in order to cohere with the work of Choi-Haslhofer-Hershkovits-White~\cite{ChoiHaslhoferHershWhite19}. We use full limits for simplicity, although it is possible to replace the full limit assumption with a subsequence, using the fact that the outer flow coincides with a unit regular Brakke flow~\cite[Theorem B6]{HershkovitsWhite17} and applying Colding-Minicozzi-Ilmanen and Colding-Minicozzi's uniqueness results~\cite{ColdingIlmanenMinicozzi16, ColMin15}.

The level set flow of $M_0$ has \textbf{three-convex blow-up type} if for each backwardly singular point $X$, $\tau(X)>0$, there exists $\la_i \to \infty$ such that $\cU_{X,\la_i}$ converges smoothly with multiplicity one to a round shrinking cylinder $B^{n+1-k} \times \bR^k$ for $k \in \{0, 1, 2\}$. In this definition, we allow for a subsequence $\la_i \to \infty$. 

The definition of cylindrical, or generic, singularities for level set flow is the same as that of three-convex blow-up type, though any multiplicity one cylindrical limit is admissible.

Finally, we say that the level set flow of $M_0$ has \textbf{mean convex neighborhoods} of singularities for $t<T$ if for each backwardly singular point of $F_t^{\text{out}}(M_0)$ and each $t \in (0,T)$, there exists $\eps>0$ such that either 
$$F_{t_2}(\ov{D_0})\cap B_{\eps}(x)\subseteq F_{t_1}(\ov{D_0})\setminus F_{t_1}^{\text{out}}(M_0)$$
for all $t-\eps \leq t_1 \leq t_2 \leq t$, or
$$F_{t_1}(\ov{D_0})\cap B_{\eps}(x)\subseteq F_{t_2}(\ov{D_0})\setminus F_{t_2}^{\text{out}}(M_0)$$
for all $t-\eps \leq t_1 \leq t_2 \leq t$. These are the flows with only backwardly singular points of mean convex and mean concave type in ~\cite{HershkovitsWhite17}.

\subsection{Singularities and Convergence}\label{section singularities mean curvature flow}\hfill\\
\vspace{-.2in}

In this paper, we say that a sequence $\mu^i$ converges as Radon measures to $\mu$ if for each $\phi \in C_c(\bR^{n+k}, \bR_{\geq 0})$, $\lim_{i \to \infty} \mu^i(\phi) = \mu(\phi)$. This is the weak$^*$ convergence of Radon measures. Likewise, a sequence $V_{\mu^i}$ converges as varifolds to $V_{\mu}$ if the sequence converges as Radon measures on $G_n(\bR^{n+k})$.

We say that a sequence of Brakke flows $\{\mu^i_t\}_{i \geq 1}$ \textbf{converges} to a Brakke flow $\mu_t$ if for each $t$, $\mu^i_t \to \mu_t$ as Radon measures, and for a.e.\ $t$, there is a subsequence $j$ depending on $t$, such that $V_{\mu^j_t} \to V_{\mu_t}$ as varifolds. We will oftentimes abbreviate a sequence of flows $\{\mu^i_t\}_{i \geq 1}$ by just $\mu^i_t$. The Brakke compactness theorem says that if each $\mu^i_t$ has uniformly locally bounded mass over the sequence, then there exists a subsequence $j$ such that $\mu^j_t$ converges to an integral Brakke flow $\mu_t$~\cite[Chp. 4]{Brakke1978}~\cite[7.1]{Ilmanen94}.

\begin{defn}\label{definition bounded area ratios}
An $n$-dimensional Brakke flow $\mu_t$ has \textbf{bounded area ratios} if 
$$\sup_{t \geq 0} \sup_{R >0} \sup_{x \in \bR^{n+k}} \frac{\mu_t(B_r(x))}{\omega_n R^n} = \La < \infty$$
\end{defn}
\noindent For an integral $n$-dimensional Brakke flow $\mu_t$, we define
$$\Te(x_0, t_0) := \lim_{t \nearrow t_0} \frac{1}{(4\pi (t_0-t))^{\frac{n}{2}}}\int \exp\Big(\frac{-|x-x_0|^2}{4(t_0-t)}\Big)\,d\mu_t(x)$$
for each spacetime point $(x_0, t_0)$. $\Te(x_0, t_0)$ is the Gaussian density of $\mu_t$ at $(x_0, t_0)$. Huisken's monotonicity formula for Gaussian density ratios of Brakke flows implies the next theorem, Theorem \ref{theorem huisken monotonicity}~\cite{Huisken90} (cf.~\cite[Lemma 7]{Ilmanen95}).
\begin{thm}[{\cite[Lem.\ 7]{Ilmanen95}}~{\cite[Cor.\ 3.2]{BernsteinWang16}}]\label{theorem huisken monotonicity}
Let $\{\mu_t\}_{t \geq 0}$ be an integral $n$-dimensional Brakke flow. Suppose that
$$\sup_{R >0} \sup_{x \in \bR^{n+k}} \frac{\mu_0(B_r(x))}{\omega_n R^n}< \infty$$
Then, $\mu_t$ has bounded area ratios.


\end{thm}
 The assumptions of Theorem \ref{theorem huisken monotonicity} are satisfied if $\mu_0 = \cH^n \lfloor M_0$ with $M_0$ a smooth closed hypersurface. Every Brakke flow in this paper will have bounded area ratios.

For a spacetime point $X = (x_0, t_0)$ with $t_0>0$ and an integral Brakke flow $\{\mu_t\}_{t \geq 0}$ with bounded area ratios, define the rescaled Brakke flow 
$$\mu_t^{\la, X}(E):=\la^n \mu_{\la^{-2}t + t_0}(\la^{-1}E+x_0)$$ 
for each $\la>0$. For a sequence $\la_i \to \infty$, there exists a subsequence such that $\mu_t^{\la_i, X}$ converges to a Brakke flow $\mu_t'$ defined for $t \in (-\infty, \infty)$. A \textbf{tangent flow} at a point $X$, with $\tau(X)>0$ is one such $\mu'_t$. There may be $T< \infty$ such that $\mu'_t \equiv 0$ for $t >T$. A \textbf{limit flow} at $X$, $\tau(X)>0$, is defined to be a limit of $\mu_t^{\la_i, X_i}$ for $\la_i \to \infty$ and $X_i \to X$.

A tangent flow is a self-similar Brakke flow for $t<0$~\cite{Ilmanen93}. A well-known folklore result says that tangent flows with smooth support have constant multiplicity on each connected component. This follows from the constancy theorem for integral varifolds with smooth support and bounded generalized mean curvature~\cite[Lemma A.1]{BellettiniWickramasekera18}. A \textbf{static} multiplicity $m$ planar tangent flow is a tangent flow given by $m\cH^{n}\lfloor P$ for some fixed hyperplane $P$ passing through the origin for $t \in (-\infty, \infty)$. A \textbf{quasistatic} multiplicity $m$ planar tangent flow is a non-static tangent flow given by $m\cH^{n}\lfloor P$ for some fixed hyperplane $P$ passing through the origin for $t \in (-\infty, 0)$. The Gaussian density $\Te$ is constant for planar tangent flows for $t<0$, and $\Te \equiv m$ for a multiplicity $m$ planar tangent flow. A quasistatic multiplicity $m$ planar tangent flow may drop in density or disappear at $t=0$. In general, a quasistatic limit flow is any Brakke flow which is static, i.e. a fixed point of the flow, until $t=0$ but not after.

We say that a Brakke flow has \textbf{generic} or cylindrical singularities if every tangent flow with $\Te>1$ is a multiplicity one shrinking cylinder $\sqrt{-2(n-k)t}\,S^k \times \bR^{n}$, for $k \in \{1, \dots, n-1\}$. These are the \textquotedblleft generic" singularities, since the cylinders are entropy generic~\cite{ColdingMinicozziGeneric}.

We define the \textbf{spacetime support} of the flow $\mu_t$ by
$$\spt \mu := \ov{\bigcup_{t \geq 0} \spt \mu_t \times \{t\}}$$
where the closure is taken in spacetime. We define the time slices of the spacetime support by
$$(\spt \mu)_{t_0}:= \spt \mu \cap \{t = t_0\}$$
Note that $\spt \mu_t \times \{t\}$ may in general be different than $(\spt \mu)_t$. A spacetime point $(x,t)$, with $t>0$, belongs to $\spt \mu$ if and only if $\Te(x,t) \geq 1$. The Gaussian density is upper semi-continuous along converging Brakke flows. One consequence is that if $\mu^i_t \to \mu_t$ as Brakke flows, then $\spt \mu^i \to \spt \mu$ as sets.

A \textbf{fully smooth point} is a spacetime point $(x_0,t_0)$ such that for some $r,\epsilon>0$ and each $s \in (t_0 - \eps^2, t_0 + \eps^2)$, 
\begin{equation}\label{equation fully smooth}\mu_s \lfloor B_r(x_0) = \cH^{n}\lfloor M_s\end{equation}
for a smooth proper mean curvature flow of embedded connected hypersurfaces $M_s$ in $B_{r}(x_0)$ with bounded curvature for $s \in (t_0 - \eps^2, t_0 + \eps^2)$. Each $M_s$ is assumed nonempty.

We stress that in the definition of fully smooth, the Brakke flow is unit density in $B_r(x_0)$. Brakke flows may disappear suddenly, but smooth mean curvature flows do not. If $(x_0, t_0)$ is a fully smooth point, then $\mu_t$ does not disappear locally around $(x_0,t_0)$. The flow $\mu_t$ is \textbf{unit regular} if each $(x,t)$ with $\Te(x,t)=1$ is a fully smooth point. If $\Te(x,t) = 1$, the only possible tangent flows at $(x,t)$ in general are a static or quasistatic multiplicity one plane. A unit regular Brakke flow has no quasistatic multiplicity one planar tangent flows.

\begin{rmk}\label{remark mult 1 static plane is fully smooth}
    A point $(x,t)$ is fully smooth if and only if there is a static multiplicity one planar tangent flow at $(x,t)$. This follows from the Brakke regularity theorem~\cite[Theorem 6.11]{Brakke1978}. See \cite[Proposition 3.7]{BernsteinWang16}, for a relevant formulation of the regularity theorem, as well as \cite[Theorem 3.6]{Ton14} for a proof of Brakke regularity. In \cite[Theorem 3.6]{Ton14}, it is proven that the motion law is satisfied where the spacetime support is smooth, which is required for our definition of fully smooth.\footnote{It is required for our definition but not actually necessary for our proofs, except out of convenience. We stipulate that a fully smooth point $x$ has a smooth mean curvature flow around $x$, but it may as well be any smooth flow of embeddings. In light of the generality of Kasai and Tonegawa's regularity theory, this means the arguments in this paper may be generalizable to other geometric flows.~\cite{kasaitonegawa14, Ton14}} See also Lemma \ref{lemma restarting}, in which we show that where the support is a smooth flow, the Brakke flow is the standard surface measure associated to the smooth flow, up to multiplicity. 
\end{rmk}

For an integral Brakke flow $\mu_t$ with bounded area ratios, define
$$\reg^+\hspace{-.04in}\mu := (\spt \mu_0 \times \{0\}) \cup \{(x,t) \in \spt \mu \,|\, (x,t) \text{ is a fully smooth point}\}$$
$$\sing^+\hspace{-.04in}\mu := \spt \mu \setminus \reg^+\hspace{-.04in} \mu$$
We define the time-slices by
$$(\sing^+\hspace{-.04in}\mu)_{t_0} := \sing^+\hspace{-.04in}\mu \cap \{t = t_0\}$$
Note that $\spt \mu$ may be smooth around $(x,t)$, yet $(x,t)$ may still be in $\sing^+\hspace{-.04in}\mu$. This could happen if $\mu_t$ has a higher multiplicity planar tangent flow. For this reason, we notate this definition of the singular set with a plus sign, so as not to confuse it with other weaker interpretations of the singular set.

\begin{rmk}
In our notation, $(\sing^+\hspace{-.04in}\mu)_{t}$ will denote all the points in the $t$-slice of $\spt \mu$ which do not have a static multiplicity one plane as a tangent flow. 
\end{rmk}

\subsection{Elliptic Regularization}\hfill\\
\vspace{-.2in}

The precise relationship between Brakke flow and level set flow is still not fully understood. A significant advance was Ilmanen's elliptic regularization, which associates a canonical unit density Brakke flow to each nonfattening level set flow~\cite{Ilmanen94}.

\begin{thm}[{\cite[Thm.\ 11.1, 11.4]{Ilmanen94}}]\label{theorem unit density for nonfattening}
Let $E_0 \subset \bR^{n+1}$ be a bounded open set with finite perimeter.
Suppose that the level set flow with initial condition $\ov{\dd^*E_0}$ is nonfattening.

Then, there exists an integral Brakke flow $\{\ov{\mu}_t\}_{t \geq 0}$ and a bounded relatively open finite perimeter set $E\subset \bR^{n+1} \times \bR_{\geq 0}$ such that
\begin{enumerate}
    \item If $u$ is the level set flow with initial condition $E_0 = \{u(\cdot, 0) >0\}$ and $\ov{\dd^*E_0} = \{u(\cdot, 0)=0\}$, then $E = \{u>0\}$,
    \item For each $t\geq 0$,
    $\{u(\cdot, t)>0\}$ is a finite perimeter set,
    \item $\mu_t := \cH^{n} \lfloor \dd^* \{u(\cdot, t)>0\}$ is a unit density integral Brakke flow,
    \item  $\ov{\mu}_t$ is unit regular, $\ov{\mu}_0 = \mu_{0}$, and $\ov{\mu}_t \geq \mu_t$ for each $t \geq 0$.

\end{enumerate}
\end{thm}
In the terminology of \cite{Ilmanen94}, which we will use occasionally, $\ov{\mu}_t$ is referred to as an \textquotedblleft enhanced motion," and $\mu_t$ is referred to as a \textquotedblleft boundary motion."

\begin{rmk}\label{remark localization of elliptic regularization}
If the level set flow of Theorem \ref{theorem unit density for nonfattening} is nonfattening until time $T$, then the results of this theorem hold for $t<T$. This follows from the fact that there is uniqueness of matching motions until time $T$ (see the proof of \cite[11.4]{Ilmanen94}). 
\end{rmk}

\section{Nonfattening and Countable Sets of Singular Times}\label{section characterization of the main assumption}
In this section, we collect multiple statements characterizing the General Assumption and flows with countable sets of singular times. 

The first statement is a well-known locality theorem for integer rectifiable varifolds due to Sch\"atzle \cite[Corollary 4.2]{Sch09}. See also \cite{LM09} for a further discussion.

\begin{thm}[{\cite[Cor.\ 4.2]{Sch09}}~\cite{LM09}]\label{theorem Schatzle locality}
Let $V_{\mu_1}, V_{\mu_2}$ be integer rectifiable $n$-varifolds defined in the open set $U\subset \bR^{n+k}$ satisfying
\begin{enumerate}
    \item $\vec{H}_{\mu_i} \in L^2_{\loc}(\mu_i)$ for $i=1,2$,
    \item $\spt \mu_1 \cap \spt \mu_2$ is $C^2$ rectifiable,
\end{enumerate}
Then, for $\cH^n$-a.e.\ $x \in \spt \mu_1 \cap \spt \mu_2$, 
$$\vec{H}_{\mu_1}(x) = \vec{H}_{\mu_2}(x)$$
\end{thm}

The next lemma says that Brakke flows \textquotedblleft sitting on top" of a smooth flow are identical to the smooth flow, assuming their initial condition is unit density. This will come in handy for applications using Ilmanen's elliptic regularization, as this situation arises naturally.
\begin{lem}\label{lemma restarting}
Let $\{\mu_t\}_{0 \leq t \leq T}$, $\{\nu_t\}_{0 \leq t < T}$ be integral $n$-dimensional Brakke flows and let $\{M_t\}_{0 \leq t < T}$ be a flow of $n$-dimensional submanifolds defined in $B_r(x)$ satisfying the following properties:
\begin{enumerate}
    \item $M_t$ is a smooth embedded connected flow of submanifolds with bounded curvature,
    \item $\mu_t, \nu_t$ have locally bounded mass,
    \item For a.e.\ $t \in [0,T)$ and each $\phi \in C^2_c(B_r(x), \bR_{\geq 0})$, $$\nu_t(\phi) \geq \mu_t(\phi),$$
    \item For a.e.\ $t \in [0, T)$,
    $$\spt \nu_t\lfloor B_r(x) = \spt \mu_t \lfloor B_r(x) = M_t$$
\end{enumerate}
Then, for a.e.\ $t \in [0,T)$, there exist positive integers $n(t)\geq m(t)$ such that
\begin{equation}\label{equation constant density n(t)}\nu_t\lfloor B_r(x) = n(t)\cH^n \lfloor M_t, \quad \text{and} \quad \mu_t\lfloor B_r(x) = m(t)\cH^n \lfloor M_t\quad \end{equation}
Moreover, if $\nu_0\lfloor B_r(x) = \cH^n \lfloor M_0$, then $n(t) = m(t) = 1$ for all $t \in [0,T)$.
\end{lem}
\begin{proof}
Brakke flows with locally bounded mass have generalized mean curvature in $L^2_{\loc}$ for almost every time. So, $\vec{H}_{\nu_t} \in L^2_{\loc}(\nu_t)$ and $\vec{H}_{\mu_t} \in L^2_{\loc}(\mu_t)$ for a.e.\ $t \in [0, T)$. Since $\mu_t$ and $\nu_t$ are Brakke flows, $$\cB(\mu_t, \phi), \cB(\nu_t, \phi)>-\infty$$
for a.e.\ $t \in [0,T)$~\cite[7.2(iv)]{Ilmanen94}. So, $V_{\mu_t}$ and $V_{\nu_t}$ are integer rectifiable varifolds for a.e.\ $t \in [0, T)$. By assumption, $V_{\mu_t}$ and $V_{\nu_t}$ have smooth support in $B_r(x)$ coinciding with the canonical unit density varifold associated to $M_t$ for a.e.\ $t \in [0, T)$. So, by Theorem \ref{theorem Schatzle locality}, for a.e.\ $t \in [0,T)$, 
$$\vec{H}_{\mu_t} = \vec{H}_{\nu_t} = \vec{H}_{M_t}$$
$\cH^n$-a.e.\ in $B_r(x)$. Since $M_t$ is a smooth flow, it has bounded mean curvature and so $\vec{H}_{\mu_t}, \vec{H}_{\nu_t}$ are bounded in $B_r(x)$ for a.e.\ $t \in [0,T)$. We may then apply the constancy theorem for integral varifolds with smooth supports, using the fact that $M_t$ is connected~\cite{Duggan86} (see \cite[Lemma A.1]{BellettiniWickramasekera18} for a proof of this version of the constancy theorem). Thus, there exist integers $n(t), m(t)$ so that (\ref{equation constant density n(t)}) holds. Since $\nu_t \geq \mu_t$ for a.e.\ $t<T$, we have that $n(t)\geq m(t)$.

Now, suppose that $\nu_0 \lfloor B_r(x) = \cH^n\lfloor M_0$. 

Since $M_t$ is a smooth flow, $t\mapsto \cH^n\lfloor M_t(\phi)$ is continuous. Suppose there exists a sequence of times $t_i \searrow 0$ such that $\liminf_{i \to \infty} n(t_i) > 1$. Using the fact that $M_t$ is a smooth flow, this would imply that for some $\phi \in C^2_c(\bR^{n+k}, \bR_{\geq 0})$,
$$\lim_{t \searrow 0} \nu_t(\phi) > \nu_0(\phi)$$
which contradicts Lemma \ref{lemma properties of BF}. Thus, there exists $t^*>0$ such that $n(t) =1$ for a.e.\ $t \in [0, t^*)$. Using the fact that $M_t$ is a smooth flow and Lemma \ref{lemma properties of BF}, we get that 
$\nu_t = \cH^n \lfloor M_t$
for all $t \in [0, t^*)$. So $n(t)$ is defined and $n(t) = 1$ for each $t \in [0, t^*)$.
Using that $M_t$ is a smooth flow,
\begin{align*}
\cH^n \lfloor M_{t^*}(\phi) &= \lim_{s \nearrow t^*}\nu_s(\phi)\\
&\geq  \nu_{t^*}(\phi)\\
&\geq \lim_{s \searrow t^*}\nu_s(\phi)\\
&=\lim_{s \searrow t^*} n(s)\cH^n \lfloor M_{s}(\phi)\\
&\geq \lim_{s \searrow t^*} \cH^n \lfloor M_{s}(\phi)\\
& = \cH^n \lfloor M_{t^*}(\phi) 
\end{align*}
which implies that $\nu_{t^*}(\phi) = \cH^n \lfloor M_{t^*}(\phi)$. We may apply the same argument for $t=t^*$ as for $t=0$ to show that (\ref{equation constant density n(t)}) holds with $n(t)= 1$ for all $t \in [0,T)$.

Since $n(t)\geq m(t)$ for a.e.\ $t \in [0, T)$, we get that $m(t) =1$ for a.e.\ $t \in [0, T)$. By a similar argument as for $\nu$ using Lemma \ref{lemma properties of BF}, we find that $m(t) = 1$ for each $t \in [0, T)$.
\end{proof}

The next lemma is a corollary of Lemma \ref{lemma restarting} which will be used throughout this paper.

\begin{lem}\label{lemma restarting with unit density}
Suppose that $\mu_t$, $\nu_t$, and $M_t$ are as in Lemma \ref{lemma restarting}, supposing assumptions (1), (2), and (4), but not (3). Suppose that for a.e.\ $t \in [0,T)$,
$\nu_t = \cH^n \lfloor M_t$ and suppose that $\mu_t$ is unit density. 

Then, assumption (3) of Lemma \ref{lemma restarting} holds and $\mu_t = \cH^n \lfloor M_t$ for each $t \in [0, T)$. 
\end{lem}
\begin{proof}
We will first show that assumption (3) of Lemma \ref{lemma restarting} holds.

Since $\mu_t$ is a Brakke flow, $\mu_t$ is integer rectifiable for a.e.\ $t \in [0,T)$. Thus, $\mu_t(B) = \int_{B} \te_t\,d\cH^n$ for an integer-valued nonnegative $\te_t$, as in Section \ref{section preliminaries Brakke flow}. By standard facts about integer rectifiable Radon measures~\cite{simon2014introduction} (see in particular ~\cite[1.3]{Ilmanen94}), $\Te^n(\mu_t, x) = \te_t(x)$ for $\cH^n$-a.e.\ $x \in \bR^{n+k}$, where $\Te^n$ is the $n$-dimensional density as in Definition \ref{definition BF}. By the assumption that $\mu_t$ is unit density, $\Te^n(\mu_t, x) = 1$ for $\mu_t$-a.e.\ $x$. Since $\mu_t$ has the same measure zero sets as $\cH^n \lfloor \{\te_t >0\}$~\cite[1.3]{Ilmanen94}, $\Te^n(\mu_t, x) = 1$ for $\cH^n$-a.e.\ $x \in \{\te_t >0\}$. Thus, $\te_t(x) = 1$ for $\cH^n$-a.e.\ $x \in \{\te_t >0\}$. Since $\{\Te^n(\mu_t, \cdot)>0\}\subseteq \spt \mu_t$ and $\Te^n(\mu_t, x) = \te_t(x)$ for $\cH^n$-a.e.\ $x \in \bR^{n+k}$, $\{\te_t>0\} \subseteq \spt \mu_t$ up to a set of $\cH^n$ measure zero. Define 
\[
    \tilde{\te_t}(x)= 
\begin{cases}
    \te_t(x),& \text{if } x \in \{\te_t>0\}\\
    1,              & \text{if } x \in \spt \mu_t \setminus \{\te_t>0\}\\
0,& \text{otherwise}
\end{cases}
\]
By definition, $\te_t(x) = 0$ for $x \notin \{\te_t>0\}$, so $\tilde{\te_t}(x) \geq \te(x)$ for each $x \in \bR^{n+k}$. Moreover, by the discussion above, $\{\tilde{\te}_t = 1\} = \spt \mu_t$ up to a set of $\cH^n$ measure zero. Thus,
$$\mu_t(B) = \int_B \te_t\, d\cH^n \leq \int_B \tilde{\te}_t\,d\cH^n = \cH^n \lfloor \spt \mu_t (B)$$
By assumption, $\spt \mu_t = M_t$ and $\nu_t = \cH^n \lfloor M_t$ for a.e.\ $t \in [0, T)$, so for a.e.\ $t \in [0, T)$, $\mu_t \leq \nu_t$ and assumption (3) of Lemma \ref{lemma restarting} is satisfied.

Applying Lemma \ref{lemma restarting}, (\ref{equation constant density n(t)}) tells us that $\mu_t = \cH^n \lfloor M_t$ for a.e.\ $t \in [0,T)$. Since $\nu_t = \cH^n \lfloor M_t$ holds for a.e.\ $t \in [0, T)$, we get by the final claim of Lemma \ref{lemma restarting} that $\mu_t = \cH^n \lfloor M_t$ for each $t \in [0, T)$.
\end{proof}

\subsection{Properties of the General Assumption}\hfill\\
\vspace{-.2in}

Throughout this section, we will prove some properties of flows satisfying the General Assumption, generally under the assumption of Conjecture \ref{conjecture E}. 

Suppose that $(M_0, \mu_t)$ satisfies the General Assumption for $t<\infty$. Recall that in the General Assumption, we define a level set flow $u: \bR^{n+1} \times \bR_{\geq 0}$ with $\{u=0\} = M_0$ and $\{u>0\} = D_0$, where $D_0$ is the open connected bounded set whose boundary is the smooth hypersurface $M_0$. Let $w_0: \bR^{n+1} \to \bR$ be the function such that $w_0 \equiv 0$ on $D_0$ and $w_0 \equiv u(\cdot, 0)$ on the complement $D'_0$. Since $u$ is continuous, $w_0$ is too. Let $w: \bR^{n+1}\times \bR_{\geq 0}$ be the level set flow with initial condition $w(\cdot, 0) = w_0$. Then, the spacetime track of the level set flow of $\ov{D_0}$ is $\{w=0\}$. Since weak set flows remain inside level set flows~\cite{Ilmanen93}, for each $a \geq 0$, $\{u=a\}\subseteq \{w=0\}$. Thus, $\{u \geq 0\}\subseteq \{w=0\}$. On the other hand, since the level set flow is independent of the choice of function representing it, for each $a<0$, $\{w=a\} = \{u=a\}$. Since $\{w=0\}$ is disjoint from $\{w=a\} = \{u=a\}$ for $a<0$, we get that $\{w=0\}\subseteq \{u \geq 0\}$. So, we find that $\{w=0\}= \{u \geq 0\}$. This means that $\{u \geq 0\}$ coincides with the spacetime track of the level set flow of $\ov{D_0}$. Using the notation of Section \ref{subsection preliminaries of level set flow}, $F_t(\ov{D_0}) = \{u(\cdot, t)\geq 0\}$ for each $t \geq 0$. Similarly, the level set flow of $D'_0$ will coincide with $\{u \leq 0\}$, and $F_t(D'_0) = \{u(\cdot, t)\leq 0\}$ for each $t \geq 0$. Therefore, if $(M_0, \mu_t)$ satisfies the General Assumption for $t<\infty$,
\begin{equation}\label{equation level set flow of domain}\cU = \{u \geq 0\}, \quad \text{and} \quad \cU' = \{u \leq 0\}\end{equation}
using the notation of Section \ref{subsection preliminaries of level set flow}. If $(M_0, \mu_t)$ satisfies the General Assumption for $t<T$, a similar statement to (\ref{equation level set flow of domain}) follows restricted to the time slab $\{0<t<T\}$.

Throughout this section, we will often assume that $(M_0, \mu_t)$ satisfies the General Assumption for $t<T_{\disc}$. This is a technically redundant statement, since the level set flow of $M_0$ is nonfattening before time $T_{\disc}$ by definition. In other words, if $(M_0, \mu_t)$ satisfies the General Assumption for $t<T$, then we can automatically assume $T\geq T_{\disc}$ by definition of $T_{\disc}$. However, we will state it like this to emphasize that we are only considering the flow until time $T_{\disc}$.

We first prove that under the General Assumption, the spacetime support of the Brakke flow $\mu_t$ coincides with the level set flow before time $T_{\disc}$. That is, as long as Conjecture \ref{conjecture E} holds, the unit density Brakke flow constructed by elliptic regularization has support coinciding with the nonfattening level set flow.

\begin{prop}\label{proposition support is level set flow}
Suppose $(M_0, \mu_t)$ satisfies the General Assumption for $t < T_{\disc}$.

Then,
$$\spt \mu \cap \{0\leq t<T_{\disc}\} = \{u=0\}\cap \{0\leq t<T_{\disc}\}$$
where $\{u=0\}$ is the level set flow as in the General Assumption.
\end{prop}
\begin{proof}
We prove this for $T_{\disc}=\infty$. A similar proof works for $T_{\disc}<\infty$. 

By construction, $(\spt \mu)\cap \{t=0\} = \{u(\cdot, 0)=0\}$. So, we only need to prove this lemma for $t>0$. 

For any set $Z\subset \bR^{n+1}\times \bR_{\geq 0}$, define 
$$Z_+ := Z\setminus (Z \cap \{t = 0\})$$
The point of restricting to $t>0$ is so that we do not have to worry about the difference between boundaries and relative boundaries in spacetime. 
\begin{lem}\label{lemma interior is u>0}
$$(\{u >0\})_+ = (\inte\{u \geq 0\})_+$$
\end{lem}
\begin{proof}
Since $(\{u>0\})_+$ is an open set, 
\begin{equation}\label{equation u subset}
    (\{u >0\})_+ \subseteq (\inte\{u \geq 0\})_+
\end{equation}
Now, we will show the other direction of (\ref{equation u subset}).
Recall from (\ref{equation level set flow of domain}) that
$$\cU = \{u \geq 0\}, \quad \text{and} \quad \cU' = \{u \leq 0\}$$
using the notation of Section \ref{subsection preliminaries of level set flow}. The $T_{\disc} = \infty$ assumption means that for each $t_0>0$,
$$F_{t_0}(M_0)\times \{t_0\} = \dd \cU \cap \{t = t_0\} = \dd \cU' \cap \{t=t_0\}$$
So,
\begin{equation}\label{equation Tdisc 2}(\{u=0\})_+ = (\dd \{u \geq 0\})_+ = (\dd \{u \leq 0\})_+\end{equation}
Since the boundary of a set is the boundary of its complement,
\begin{equation}\label{equation Tdisc}
(\{u=0\})_+ = (\dd\{u > 0\})_+ = (\dd \{u < 0\})_+
\end{equation}

Now, let $x \in (\inte\{u \geq 0\})_+$ such that $x \in (\{u=0\})_+$. By assumption, there exists a small $r>0$ such that 
$$B^{n+2}_r(x)\subset (\{u \geq 0\})_+$$
This implies that 
$$x \notin (\dd\{u <0\})_+$$
Since $x \in (\{u=0\})_+$, this contradicts (\ref{equation Tdisc}). Therefore, if $x \in (\inte\{u \geq 0\})_+$, then $x \notin (\{u =0\})_+$ and so $x \in (\{u>0\})_+$. This proves 
\begin{equation}\label{equation u superset}(\inte\{u \geq 0\})_+ \subseteq (\{u>0\})_+\end{equation}
Combining (\ref{equation u subset}) with (\ref{equation u superset}), we conclude the lemma.
\end{proof}

Recall that $(\{u >0\})_+$ is an open set of finite perimeter by Theorem \ref{theorem unit density for nonfattening}. Also, recall that $F_t(M_0)$ is nonfattening by assumption, since $T_{\disc}< T_{\fat}$. This means that 
\begin{equation}\label{equation nonfattening in n+1 dimensions}
\cH^{n+2}(\{u=0\})=0
\end{equation}
by \cite[Lemma 11.3]{Ilmanen94}. So, (\ref{equation nonfattening in n+1 dimensions}) implies that $(\{u \geq 0\})_+$ is a set of finite perimeter as well~\cite[Remark 12.4]{Mag12}.

Now,
\begin{align}\label{equation support calculation}
    (\spt \mu)_+ &= (\ov{\bigcup_{t \geq 0} \spt \mu_t \times \{t\}})_+\nonumber\\
    &=\big(\ov{\bigcup_{t \geq 0} \ov{\dd^*\{u(\cdot, t)>0\}}\times \{t\}}\big)_+\nonumber\\
\intertext{By a rephrasing of Ilmanen's lemma that the support of the boundary of the current associated to $\{u>0\}$ is the closure of the union of the supports of all slices of the boundary current~\cite[Lemma 11.6]{Ilmanen94} (see also \cite[Theorems 11.2, 11.4]{Ilmanen94}),}   
    &= (\ov{\dd^*\{u > 0\}})_+\nonumber\\
\intertext{Since two sets of finite perimeter have the same reduced boundary if they differ by a set of measure zero~\cite[Remark 15.2]{Mag12},}
    &= (\ov{\dd^*\{u \geq 0\}})_+\nonumber\\
\intertext{For closed sets of finite perimeter, the closure of the reduced boundary is the topological boundary of the interior~\cite[Theorem 4.4]{Giu84}, so}
    &= (\dd\inte\{u \geq 0\})_+\nonumber\\
\intertext{By Lemma \ref{lemma interior is u>0},}
    &= (\dd \{u > 0\})_+\nonumber\\
\intertext{By (\ref{equation Tdisc}),}
&= (\{u=0\})_+\nonumber
\end{align}
We conclude that $(\spt \mu)_+ = (\{u=0\})_+$. This completes the proof of Proposition \ref{proposition support is level set flow} in the $T_{\disc} = \infty$ case. The $T_{\disc}< \infty$ case follows similarly. 
\end{proof}

We now show that a generic level set flow is nonfattening and satisfies Conjecture \ref{conjecture E}. This will be used to prove Corollary \ref{corollary mean convex neighborhoods of singularities}.
\begin{prop}\label{proposition general assumption is generic}
A generic level set flow satisfies $T_{\disc} = T_{\fat} = \infty$. 

Let N be a smooth closed embedded hypersurface, and let $f: N \times [-1,1] \to \bR^{n+1}$ be a smooth one-parameter family of embedded graphs $f(\cdot, s)$ over $N$. Then, there is a full measure set $J \subseteq [-1,1]$ such that for $s \in J$, the level set flow of $f(N,s)$ satisfies 
$$T_{\disc} = T_{\fat} = \infty$$ 
\end{prop}
\begin{proof}
It is well-known that a generic level set flow is nonfattening~\cite[Appendix C]{HershkovitsWhite17}. Ilmanen proved that if $g$ is a Lipschitz function on $\bR^n$ with compact level sets and $u$ is the level set flow with initial condition $g$, then for almost every $a \in \bR$, 
\begin{equation}\label{equation ilmanen 12.11}(\{u=a\})_+ = (\ov{\dd^* \{u > a\}})_+\end{equation} 
where $\{u>a\}$ is bounded~\cite[12.11]{Ilmanen94}. Here, we use the notation $Z_+$ to denote the restriction of the spacetime set $Z$ to positive times, as in Proposition \ref{proposition support is level set flow}. A similar statement also holds for $\{u<a\}$ by applying (\ref{equation ilmanen 12.11}) to the complement of the compact domain bounded by $N$ intersected with a large ball containing $N$. Choose a function $g_0$ with compact level sets such that for each $s \in [-1,1]$, $f(N,s)$ is a level set of $g_0$. Then, let $u$ be the level set flow with initial condition $g_0$. By \cite[12.11]{Ilmanen94}, for a.e.\ $a \in \bR$, $\{u=a\}$ is nonfattening and
\begin{equation}\label{equation ilmanen 12.11 both sides}(\{u=a\})_+ = (\ov{\dd^*\{u<a\}})_+ = (\ov{\dd^*\{u>a\}})_+\end{equation}

Since $\{u=a\}$ is nonfattening, $T_{\fat} = \infty$.
Recall that each $\{u<a\}$ and $\{u>a\}$ is a set of finite perimeter by elliptic regularization (see Theorem \ref{theorem unit density for nonfattening}). Using basic facts about sets of finite perimeter (see the proof of Proposition \ref{proposition support is level set flow}),
\begin{align*}
    (\{u=a\})_+ &= (\ov{\dd^*\{u>a\}})_+\\
    &= (\ov{\dd^*\{u\geq a\}})_+\\
    &= (\dd \inte \{u\geq a\})_+
\end{align*}
This works similarly for $\{u<a\}$, so (\ref{equation ilmanen 12.11 both sides}) implies
\begin{equation}\label{equation boundary interior}
(\dd \inte \{u \leq a\})_+ = (\{u=a\})_+ = (\dd \inte \{u \geq a\})_+
\end{equation}
Let $x \in (\dd \{u \geq a\})_+$. The continuity of $u$ implies that $x \in (\{u=a\})_+$. Applying (\ref{equation boundary interior}), $x \in (\dd \inte \{u \geq a\})_+$ and so 
$$(\dd \{u \geq a\})_+ \subseteq (\dd \inte \{u \geq a\})_+$$
Let $y \in (\dd \inte \{u \geq a\})_+$. Then, for each $r>0$, $B^{n+2}_r(y)$ intersects $(\{u<a\})_+$. If not, then $y \in \inte\{u \geq a\}$, which is impossible since the boundary of an open set is disjoint from the open set. Thus, since $B^{n+2}_r(y)$ intersects $(\{u<a\})_+$ for all small $r>0$, $y \in (\dd \{u \geq a\})_+$. This implies that
$$(\dd \inte \{u \geq a\})_+ \subseteq (\dd \{u \geq a\})_+$$
and so 
$$(\dd \{u \geq a\})_+ = (\dd \inte \{u \geq a\})_+$$
Arguing similarly for $\{u \leq a\}$ using (\ref{equation boundary interior}), we get that
$$(\dd \{u \leq a\})_+ = (\{u=a\})_+ = (\dd \{u \geq a\})_+$$
This implies that $T_{\disc} = T_{\fat} = \infty$, since $\{u \geq a\}$ and $\{u \leq a\}$ are $\cU$ and $\cU'$, respectively, by (\ref{equation level set flow of domain}).
\end{proof}

The next proposition says that for $t<T_{\disc}$, the support of $\mu_t$ coincides with the level set flow at almost every time. Note that there is a difference in our notation between $(\spt \mu)_t$ and $\spt \mu_t \times \{t\}$. It follows immediately from Proposition \ref{proposition support is level set flow} that $(\spt \mu)_t = F_t(M_0) \times \{t\}$ for each $t<T_{\disc}$.
\begin{prop}\label{proposition a.e. t spt mut is Ft}
Suppose $(M_0, \mu_t)$ satisfies the General Assumption for $t < T_{\disc}$.

Then, for a.e.\ $t \in [0, T_{\disc})$,
$$\spt \mu_t = F_t(M_0)$$
\end{prop}
\begin{proof}
We first note that 
$$\spt \mu_t = \ov{\dd^*\{u(\cdot, t)>0\}}$$
Since $\cH^{n+1}(\{u(\cdot, t)=0\})=0$ for each $t<T_{\disc}$ by the nonfattening condition and since $\{u(\cdot, t)\geq 0\}$ is a set of finite perimeter, 
$$\ov{\dd^*\{u(\cdot, t)>0\}} = \ov{\dd^*\{u(\cdot, t)\geq 0\}}$$
by \cite[Remark 12.4, 15.2]{Mag12}. Also, by definition, $F_t(\ov{D_0}) = \{u(\cdot, t)\geq 0\}$. Rewriting, we have that
$$\spt \mu_t = \ov{\dd^* F_t(\ov{D_0})}$$
Since the closure of the reduced boundary of a closed set of finite perimeter is the topological boundary of the interior~\cite[Theorem 4.4]{Giu84},
\begin{equation}\label{equation spt mut halfway}
\spt \mu_t = \dd \inte F_t(\ov{D_0})\end{equation}
\begin{lem}\label{lemma halfway spt lemma}
For a.e.\ $t \in [0, T_{\disc})$,
\begin{equation}\label{equation lemma halfway spt lemma}\dd \{u(\cdot, t)>0\} = \dd\{u(\cdot, t)<0\} = F_t(M_0)\end{equation}
\end{lem}
\begin{proof}
    By \cite[Theorem C10]{HershkovitsWhite17}, for a.e.\ $t\in [0, T_{\disc})$, $\dd F_t(\ov{D_0}) = F_t^{\text{out}}(M_0)$. Similarly, for a.e.\ $t \in [0, T)$, $F_t^{\text{in}}(M_0) = \dd F_t(D'_0)$. Using the fact that $t < T_{\disc}$, 
    $$F_t^{\text{out}}(M_0) = F_t^{\text{in}}(M_0) = F_t(M_0)$$ 
   so for a.e.\ $t \in [0,T_{\disc})$,
    $$\dd F_t(\ov{D_0}) = \dd F_t(D'_0) = F_t(M_0)$$
    This means that
    $$\dd\{u(\cdot, t)\geq 0\} = \dd\{u(\cdot, t)\leq 0\} = F_t(M_0)$$
    for a.e.\ $t<T$. Since the boundary of a set is the boundary of its complement,
    \begin{equation}\label{equation boundary u>0 is u=0}
        \dd \{u(\cdot, t)>0\} = \dd\{u(\cdot, t)<0\} = F_t(M_0)
    \end{equation}
\end{proof}

\begin{lem}\label{lemma second halfway spt}
For a.e.\ $t \in [0, T_{\disc})$,
$$\inte F_t(\ov{D_0}) = \{u(\cdot, t)>0\}$$
\end{lem}
\begin{proof}
Let $t$ satisfy (\ref{equation lemma halfway spt lemma}), the conclusion of Lemma \ref{lemma halfway spt lemma}.

Let $x \in \inte F_t(\ov{D_0})$ such that $x \in \{u(\cdot, t)=0\}$. By assumption, there exists a small $r>0$ such that 
$$B_r(x)\subset  F_t(\ov{D_0}) = \{u(\cdot, t)\geq 0\}$$
This implies that 
$$x \notin \dd\{u(\cdot, t)<0\} $$
Thus, $x \in \{u(\cdot, t)=0\}$ yet $x \notin \dd\{u(\cdot, t)<0\}$, which contradicts (\ref{equation lemma halfway spt lemma}). Therefore, if $x \in \inte F_t(\ov{D_0})$, then $x \notin \{u(\cdot, t)=0\}$ and so $x \in \{u(\cdot, t)>0\}$. This proves 
$$\inte F_t(\ov{D_0}) \subseteq \{u(\cdot, t)>0\}$$
On the other hand,
$$\{u(\cdot, t)>0\}\subseteq \inte F_t(\ov{D_0})$$
since $\{u(\cdot, t)>0\}$ is an open set. So, for each $t$ satisfying (\ref{equation lemma halfway spt lemma}), $$\inte F_t(\ov{D_0}) = \{u(\cdot, t)>0\}$$
By Lemma \ref{lemma halfway spt lemma}, we conclude this lemma.
\end{proof}
Combining Lemmas \ref{lemma halfway spt lemma} and \ref{lemma second halfway spt} with (\ref{equation spt mut halfway}), for a.e.\ $t \in [0, T_{\disc})$,
\begin{align*}\label{equation spt mut halfway}
\spt \mu_t &= \dd \inte F_t(\ov{D_0})\\
&= \dd\{u(\cdot, t)>0\}\\
&= F_t(M_0)
\end{align*}
This concludes the proposition.
\end{proof}

Under the General Assumption and Conjecture \ref{conjecture E}, we rule out quasistatic multiplicity one planes as tangent flows of $\mu_t$. This means that the Brakke flow $\mu_t$ is unit regular until at least time $T_{\disc}$. This is somewhat similar to the work of Hershkovits-White in~\cite[Theorem B.9]{HershkovitsWhite17}.

Recall that the enhanced motion, $\ov{\mu}_t$, constructed by elliptic regularization is unit regular (see Theorem \ref{theorem unit density for nonfattening}). However, enhanced motions may be distinct from the boundary motion $\mu_t$ which \textquotedblleft sits underneath" the enhanced motion. In Proposition \ref{proposition no quasistatic mult 1 planes}, we give a general condition under which the boundary motion $\mu_t$ is unit regular. 

\begin{prop}\label{proposition no quasistatic mult 1 planes}
Suppose $(M_0, \mu_t)$ satisfies the General Assumption for $t < T_{\disc}$.

Then, no tangent flow of $\mu_t$ at a spacetime point $X$, with $0<\tau(X)<T_{\disc}$, is a quasistatic multiplicity one plane. 
\end{prop}
\begin{proof} Let $T = T_{\disc}$ and let $t_0 \in (0, T)$.

Suppose that $(x,t_0) \in (\spt \mu)_{t_0}$ is a spacetime point with a quasistatic multiplicity one planar tangent flow. Then, by the Brakke regularity theorem, there exists a ball $B_{\eps}(x)$ and a time interval $s \in (t_0 - \eps^2, t_0)$ such that 
\begin{equation}\label{equation support is smooth}\spt \mu_s\lfloor B_{\eps}(x) = M_s\end{equation}
with $M_s \subset B_{\eps}(x)$ a smooth proper flow of embedded connected hypersurfaces with uniformly bounded curvature over $(t_0 - \eps^2, t_0)$. See \cite[Proposition 3.7]{BernsteinWang16} for a version of the Brakke regularity theorem that implies this statement. Since $\mu_s$ is unit density, (\ref{equation support is smooth}) implies by Lemma \ref{lemma restarting with unit density} that
$$\mu_s \lfloor B_{\eps}(x) = \cH^n \lfloor M_s$$
for $s \in (t_0 - \eps^2, t_0)$.

By Proposition \ref{proposition a.e. t spt mut is Ft}, we may choose $s_0 \in (t_0 - \eps^2, t_0)$ such that 
$$\ov{\dd^*\{u(\cdot, s_0)>0\}} = \spt \mu_{s_0} = F_{s_0}(M_0)$$
Let $\nu_t$ be the unit regular Brakke flow constructed via elliptic regularization with initial condition $F_{s_0}(M_0)$ (see Theorem \ref{theorem unit density for nonfattening}) so that $\nu_{s_0} = \mu_{s_0}$. By Theorem \ref{theorem unit density for nonfattening}, the Brakke flow $\nu_t$ exists for $t\geq s_0$, and 
$$\spt \mu \cap \{s_0 < t< T\} \subseteq \spt \nu \cap \{s_0 < t< T\}$$
by the fact that $\mu_t(\phi) \leq \nu_t(\phi)$ for each $\phi \in C^2_c(\bR^{n+1}, \bR_{\geq 0})$.

Proposition \ref{proposition support is level set flow} says that
$$\spt \mu \cap \{s_0 < t< T\} = \{u=0\}\cap \{s_0 < t< T\}$$
and since codimension one Brakke flows are contained in the level set flow~\cite[Chp. 10]{Ilmanen94},
$$\spt \nu \cap \{s_0 < t< T\} \subseteq \{u=0\}\cap \{s_0 < t< T\}$$
This implies that
\begin{equation}\label{equation equality of supports of mu nu}
\{u=0\}\cap \{s_0 < t< T\} = \spt \mu \cap \{s_0 < t< T\} = \spt \nu \cap \{s_0 < t< T\}
\end{equation}
By (\ref{equation equality of supports of mu nu}), we have that for $t \in (s_0, T)$, $F_t(M_0) = (\spt \nu)_t$. By Theorem \ref{theorem unit density for nonfattening}, $\nu_t \geq \mu_t$ and so $\spt \mu_t \subseteq \spt \nu_t$ for each $t \geq s_0$. So, for $t \in (s_0, T)$, 
$$\spt \mu_t \subseteq \spt \nu_t \subseteq (\spt \nu)_t = F_t(M_0)$$
By Proposition \ref{proposition a.e. t spt mut is Ft}, $\spt \mu_t = F_t(M_0)$ for a.e.\ $t \in (0, T)$ which implies that $\spt \mu_t = \spt \nu_t$ for a.e.\ $t \in (s_0, T)$. Thus, by (\ref{equation support is smooth}), 
\begin{equation}\label{equation support of elliptic reg coincides}
M_s = \spt \mu_s \lfloor B_{\eps}(x) = \spt \nu_s \lfloor B_{\eps}(x)
\end{equation}
for a.e.\ $s \in (t_0 - \eps^2, t_0)$. Applying Lemma \ref{lemma restarting} to $\nu_s\lfloor B_{\eps}(x)$ and $\mu_s\lfloor B_{\eps}(x) = \cH^{n}\lfloor M_s$ over the time interval $[s_0, t_0)$ and using the fact that $\mu_t \leq \nu_t$ and $\mu_{s_0} = \nu_{s_0}$, we get that 
$$\nu_s \lfloor B_{\eps}(x) = \cH^{n} \lfloor M_s$$ for $s \in [s_0, t_0)$. 

Since $\Te(x,t_0)=1$ for the Brakke flow $\cH^{n} \lfloor M_s$, we get that $\Te(x,t_0)=1$ for the flow $\nu_t$ as well. 

By Theorem \ref{theorem unit density for nonfattening}, $\nu_t$ is unit regular and so $\nu_t$ has a static multiplicity one planar tangent flow at $(x,t_0)$. By Remark \ref{remark mult 1 static plane is fully smooth}, $(x,t_0)$ is a fully smooth point for $\nu_t$. Therefore, for some $\de>0$, $\nu_s \lfloor B_{\de}(x) = \cH^n \lfloor \tilde{M}_s$ for $s \in (t_0 - \de^2, t_0+\de^2)$, where $\tilde{M}_s$ is a smooth proper flow of smooth embedded connected hypersurfaces with uniformly bounded curvature in $B_{\de}(x)$ defined for $(t_0 - \de^2, t_0+\de^2)$ and $\tilde{M}_s = M_s$ for $s \in (t_0 - \de^2, t_0)$.

By (\ref{equation equality of supports of mu nu}),
\begin{equation}\label{equation equality of supports of mu nu 2}\bigcup_{s \in (t_0 - \de^2, t_0 + \de^2)}\tilde{M}_s = \{u=0\}\cap (B_{\de}(x) \times (t_0 - \de^2, t_0 + \de^2))\end{equation}
By Proposition \ref{proposition a.e. t spt mut is Ft}, $\spt \mu_t = F_t(M_0)$ for a.e.\ $t \in [0,T)$, so (\ref{equation equality of supports of mu nu}) and (\ref{equation equality of supports of mu nu 2}) imply
$$\spt \mu_s\lfloor B_{\de}(x)= \tilde{M}_s$$
for a.e.\ $s \in (t_0 - \de^2, t_0 + \de^2)$.

Applying Lemma \ref{lemma restarting with unit density} to $\mu_s\lfloor B_{\de}(x)$ and $\cH^{n}\lfloor \tilde{M}_s$ for $s \in (t_0 - \de^2, t_0 + \de^2)$ and using the fact that $\mu_s$ is unit density,
$$\mu_s\lfloor B_{\de}(x) = \cH^{n} \lfloor \tilde{M}_s$$
for $s \in (t_0 - \de^2, t_0 + \de^2)$.
Thus, $(x,t_0)$ is a fully smooth point for the Brakke flow $\mu_{t}$. That is, there is a static multiplicity one planar tangent flow of $\mu$ at $(x,t_0)$. By the uniqueness of static multiplicity one planar tangent flows, if one tangent flow is a static multiplicity one plane, then they all are. This contradicts the assumption that there is a quasistatic multiplicity one planar tangent flow at $(x,t_0)$. 

Since $t_0$ was arbitrary, no tangent flows are quasistatic multiplicity one planes for $t \in (0, T_{\disc})$.
\end{proof}

We will now relate blow-ups of level set flow to blow-ups of Brakke flow under a $T_{\disc}$ assumption. Recall the definitions from Section \ref{subsection preliminaries of level set flow}.

\begin{lem}\label{lemma lsf blowups are bf blowups}
Let $\cU$ be the spacetime track of the outer flow for a level set flow with smooth closed initial condition $M_0$. Let $X$ be a backwardly singular spacetime point of the outer flow, with $0< \tau(X) < T_{\disc}$, such that there exists a subsequence $\la_i \to \infty$ so that $\cU_{X,\la_i}$ converges smoothly with multiplicity one to a round shrinking solid cylinder $B^{n+1-k} \times \bR^k$.

If $\mu_t$ is the unit density Brakke flow associated to the level set flow for $0\leq t<T_{\disc}$, then each tangent flow of $\mu_t$ at $X$ is a multiplicity one round shrinking cylinder $S^{n-k} \times \bR^k$.
\end{lem}
\begin{proof}
Let $\mu_t$ be the canonical unit density Brakke flow associated to the level set flow of $M_0$ for $0\leq t< T_{\disc}$. 
By Proposition \ref{proposition support is level set flow}, 
\begin{equation}\label{equation support coincides}\dd\cU \cap \{0\leq t<T_{\disc}\} = \spt \mu \cap \{0\leq t< T_{\disc}\}\end{equation}
Recall from Section \ref{section singularities mean curvature flow} that $\mu^{X, \la_i}$ denotes the parabolic rescaling of the Brakke flow $\mu_t$ around $X$ by $\la_i$. So, in each time slab $\{a<t<b\}$, (\ref{equation support coincides}) implies $\dd\cU_{X, \la_i}=\spt \mu^{X,\la_i}$ for large $\la_i$. By the Brakke compactness theorem, there exists a subsequence $\la_j \to \infty$ such that $\mu^{X, \la_j}$ converges to a tangent flow $\mu'$. In each compact set, the supports of converging Brakke flows converge as sets to the support of the limit, due to the upper semi-continuity of the Gaussian density. So, in each compact set, $\spt \mu'$ is the limit of $\spt \mu^{X, \la_j}$ as sets. By assumption, $\cU_{X, \la_j}$ converge as sets to a shrinking $B^{n+1-k}\times \bR^k$ and $\dd \cU_{X, \la_j}$ converges smoothly with multiplicity one to a shrinking $S^{n-k}\times \bR^k$. Since $\dd\cU_{X, \la_j} = \spt \mu^{X, \la_j}$ in each time slab $\{a<t<b\}$, we find that $\spt \mu^{X, \la_j}$ converges smoothly with multiplicity one to a shrinking cylinder $S^{n-k}\times \bR^k$. Any limit of $\spt \mu^{X, \la_j}$ as sets must be a dense subset of the shrinking cylinder $S^{n-k}\times \bR^k$. Since $\spt \mu'$ is closed and is the limit of $\spt \mu^{X, \la_j}$ as sets, $\spt \mu'$ is the shrinking cylinder $S^{n-k}\times \bR^k$. 

Since $\mu_t$ is unit density and $\dd \cU_{X, \la_j} = \spt \mu^{X, \la_j}$ converges smoothly with multiplicity one to the shrinking $S^{n-k}\times \bR^k$, the tangent flow $\mu'_t$ is a multiplicity one shrinking $S^{n-k}\times \bR^k$. If one tangent flow is a multiplicity one cylinder, then all tangent flows are cylindrical by work of Colding-Ilmanen-Minicozzi~\cite{ColdingIlmanenMinicozzi16}, and in fact the axis of the cylinder is unique by work of Colding-Minicozzi~\cite{ColMin15}. This holds for unit regular Brakke flows with bounded Gaussian density ratios. Unit regularity is ensured by Proposition \ref{proposition no quasistatic mult 1 planes}, and Gaussian density ratios are bounded by the fact that area ratios are bounded (see Theorem \ref{theorem huisken monotonicity}). We find that each tangent flow of $\mu_t$ at $X$ is a multiplicity one shrinking cylinder $S^{n-k}\times \bR^k$.
\end{proof}

We conclude this subsection by showing that for nonfattening level set flows with only cylindrical singularities, the associated unit density Brakke flow has only cylindrical tangent flows. The point of the next proposition is that for $t<T_{\disc}$, backwardly regular points of the level set flow correspond to static multiplicity one planar tangent flows. In particular, backwardly regular points do not correspond to higher multiplicity planar tangent flows due to the unit density of $\mu_t$. In this proposition, we will use the same notation as Lemma \ref{lemma lsf blowups are bf blowups}. 

\begin{prop}\label{proposition no Tmult for lsf with cylindrical singularities}
Suppose $(M_0, \mu_t)$ satisfies the General Assumption for $t< T_{\disc}$, and suppose that the level set flow of $M_0$ has only cylindrical singularities for $t<T_{\disc}$. That is, for each backwardly singular spacetime point $X$ of the outer flow, with $0<\tau(X)<T_{\disc}$, $\cU_{X,\la_i}$ converges smoothly with multiplicity one to a round shrinking solid cylinder for some $\la_i \to \infty$.

Then, each nonempty tangent flow of $\mu_t$, for $0<t<T_{\disc}$, is a static multiplicity one plane or a multiplicity one round shrinking cylinder.
\end{prop}
\begin{proof}
If $X$, with $0<\tau(X) < T_{\disc}$, is a backwardly singular spacetime point of the outer flow of $M_0$, then Lemma \ref{lemma lsf blowups are bf blowups} implies that $\mu_t$ has only multiplicity one round shrinking cylindrical tangent flows at $X$.

Suppose that $X = (x,t)$, $t<T_{\disc}$, is not a backwardly singular point of the outer flow. Then, $X$ is a backwardly regular point of the outer flow. So, there is a smooth mean curvature flow of smooth embedded connected hypersurfaces $M_s$ defined in $B_r(x)$ for $s \in (t - \eps^2, t]$ such that $F_s^{\text{out}}(M_0) \cap B_r(x) = M_s$ for $s \in (t - \eps^2, t]$. Since $t<T_{\disc}$, the level set flow of $M_0$ coincides with the outer flow until time $T_{\disc}$, so $F_s(M_0) \cap B_r(x) = M_s$ for $s \in (t - \eps^2, t]$.  By Proposition \ref{proposition a.e. t spt mut is Ft}, for a.e.\ $s\in (t - \eps^2, t]$, $\spt \mu_s\lfloor B_r(x) = M_s$. Since the flow $\mu_t$ is unit density, Lemma \ref{lemma restarting with unit density} implies that $\mu_s\lfloor B_r(x) = \cH^n \lfloor M_s$ for $s \in (t- \eps^2, t]$. Then, since $M_s$ is a smooth flow of embedded hypersurfaces for $s \in (t - \eps^2, t]$, we have that $\Te(x,t) = 1$. Since $t<T_{\disc}$, we know that $\mu_t$ has no quasistatic multiplicity one planes as tangent flows by Proposition \ref{proposition no quasistatic mult 1 planes}. Thus, $\Te(x,t) = 1$ implies that the tangent flow of $\mu_t$ at $X$ is a static multiplicity one plane. So the only tangent flows of $\mu_t$ at any backwardly regular point of the outer flow is a static multiplicity one plane. 

Since the outer flow of $M_0$ coincides with the level set flow of $M_0$ for $t<T_{\disc}$, $\spt \mu$ coincides with the outer flow for $t< T_{\disc}$ by Proposition \ref{proposition support is level set flow}. We have classified the tangent flows for all spacetime points $X$ in the outer flow, so we have found all tangent flows of $\mu_t$ at spacetime points of $\spt \mu$. This concludes the proposition. 
\end{proof}

\subsection{Countable Singular Times}\hfill\\
\vspace{-.2in}

In this section, we will analyze the consequences of the assumption that there is a countable set of singular times and provide natural examples for which this holds.

A Brakke flow $\mu_t$ has a \textit{countable set of singular times} for $t<T$ if there is a cocountable set $\cI \subseteq [0, T)$ such that for each $t_0 \in \cI$, $\Te(x,t_0)\leq 1$ for each $x \in \bR^{n+k}$. Note that if $\Te(x,t_0)<1$ and $t_0>0$, then $(x,t_0) \notin (\spt \mu)_{t_0}$.

A level set flow $u$ has an \textit{open cocountable set of regular times} for $t<T$ if there is a relatively open cocountable set $\cI \subseteq [0, T)$ such that for each $t_0 \in \cI$, $F_{t_0}(M_0) = \{u(\cdot, t_0)=0\}$ is a smooth embedded, possibly disconnected, hypersurface in $\bR^{n+1}$.

\begin{prop}\label{proposition smoothness a.e. time}
Suppose $M_0$ is compact, and the level set flow of $M_0$ has an open cocountable set of regular times for $t<T$.

Then, for all but countably many $t \in [0, T)$,
\begin{enumerate}
    \item There are finitely many connected components of $F_t(M_0)$,
    \item There exists an interval $(t_1, t_2)$ containing $t$ such that $\{F_s(M_0)\}_{s \in (t_1, t_2)}$ is a smooth flow of closed, properly embedded hypersurfaces.
\end{enumerate}
\end{prop}
\begin{rmk}
Although the number of connected components of $F_t(M_0)$ is finite for all but countably many times, there is no control on this number. 
\end{rmk}
\begin{proof}\hfill\\
\noindent\textbf{Proof of (1):}

Let $\cI\subset [0, T)$ be the relatively open cocountable set such that for each $t_0 \in \cI$, $F_{t_0}(M_0)$ is a smooth embedded hypersurface. Let $t_0 \in \cI$. Since $F_{t_0}(M_0)$ is a closed set in $\bR^{n+1}$, $F_{t_0}(M_0)$ is a an embedding with closed image. This implies that $F_{t_0}(M_0)$ is a properly embedded hypersurface.

Suppose there exist infinitely many connected components of $F_{t_0}(M_0)$. Let $\{N_i\}_{i \geq 1}$ be a sequence of distinct connected components of $F_{t_0}(M_0)$. Let $\{x_i\}_{i \geq 1}$ be a sequence such that for each $i \geq 1$, $x_i \in N_i$. 

Since $M_0$ is a bounded set, $F_{t_0}(M_0)$ is bounded by the avoidance principle for level set flows. Also, $F_{t_0}(M_0)$ is a closed subset of $\bR^{n+1}$ for each $t_0$. So, $F_{t_0}(M_0)$ is compact, and some subsequence $x_j$ converges to $x \in F_{t_0}(M_0)$. Let $B_r(x)$ be a compact ball. Since $F_{t_0}(M_0)$ is a properly embedded hypersurface, this means that $B_r(x) \cap F_{t_0}(M_0)$ is compact in the intrinsic distance. However, each $x_i$ belongs to a distinct connected component, so $x_i$ does not converge to $x$ in the intrinsic sense. This contradicts the fact that $F_{t_0}(M_0)$ is properly embedded. 

Therefore, for each $t \in \cI$, there are finitely many connected components of $F_t(M_0)$.

\noindent \textbf{Proof of (2):}

Let $t_0 \in \cI$, where $\cI$ is the relatively open cocountable set of regular times in $[0, T)$.

By (1), $F_{t_0}(M_0)$ is the union of finitely many connected components $\{N_1, \dots, N_m\}$. By assumption, each $N_i$ is a smooth embedded hypersurface.

Let $N_i$ be a connected component of $F_{t_0}(M_0)$. Let $U \subset \bR^{n+1}$ be a smooth connected bounded open set, and suppose $N_i$ is a proper subset of $\dd U$. Then, by Evans-Spruck's instantaneous extinction theorem \cite[Theorem 8.1]{EvansSpruck91}, $F_s(N_i) = \emptyset$ for all $s > t_0$.

On the other hand, suppose $N_i = \dd U$. Then, $N_i$ is a closed properly embedded hypersurface, and there exists $s_i>t_0$ such that $F_s(N_i)$ is the smooth flow of $N_i$ for $s \in [t_0, s_i)$.

Since there are only finitely many $N_i$, there exists $s^* > t_0$ such that for each $s \in (t_0, s^*)$, each connected component of $F_s(M_0)$ is a smooth closed properly embedded hypersurface, and $F_s(M_0)$ is a smooth flow for $s \in (t_0, s^*)$. In fact, $F_s(M_0)$ will flow smoothly until $F_s(M_0)$ is not a smooth embedded hypersurface. So,
\begin{equation}\label{equation s*}
s^* = \inf\{s>t_0\,|\,s \notin \cI\}\end{equation}

Since $\cI \subseteq [0,T)$ is relatively open, $\cI$ is the countable union of disjoint relatively open intervals. Let $t_1 \in \cI$, and let $I_i$ be the interval such that $t_1 \in I_i$. Suppose without loss of generality that $t_1 \neq 0$. Then, there exists $s_* \in I_i$ such that $s_* < t_1$. Applying (\ref{equation s*}) to $t_0 = s_*$ and using the fact that $I_i$ is a relatively open interval, there exists $s^*>t_1$ such that $F_s(M_0)$ is a smooth flow of closed properly embedded hypersurfaces for $s \in (s_*, s^*)$. Since $t_1$ is any point in $\cI \cap \bR_{\geq 0}$, which is a cocountable set in $[0,T)$, and $t_1 \in (s_*, s^*)$ where $F_s(M_0)$ is a smooth flow for $(s_*, s^*)$, we conclude this proposition.

\end{proof}

\begin{cor}\label{corollary main assumption implies vanishing singular set each time}
Suppose that one of the following assumptions holds:
\begin{itemize}
    \item $(M_0, \mu_t)$ satisfies the General Assumption for $t<T$, the level set flow of $M_0$ has an open cocountable set of regular times for $t<T$, and $T=T_{\disc}$,
    \item $\mu_t$ is an integral $n$-dimensional Brakke flow in $\bR^{n+k}$ which is unit regular, has bounded area ratios, and has a countable set of singular times for $t<T$.
\end{itemize}
Then, for all but countably many $t \in [0, T)$,
$$(\sing^+\hspace{-.04in}\mu)_t = \emptyset$$
\end{cor}
\begin{proof}\hfill\\
\noindent \textbf{Level set flow assumption:}

Under the assumption on the level set flow, Proposition \ref{proposition smoothness a.e. time} says that for all but countably many $t \in [0, T_{\disc})$, there exists $(t_1, t_2)\ni t$ such that $F_s(M_0)$ is a smooth flow of closed embedded hypersurfaces for $s \in (t_1, t_2)$. By Proposition \ref{proposition a.e. t spt mut is Ft}, $\spt \mu_s = F_s(M_0)$ for a.e.\ $s<T_{\disc}$. Since $\mu_t$ is unit density, Lemma \ref{lemma restarting with unit density} applied to each connected component implies that $\mu_s = \cH^n \lfloor F_s(M_0)$ for $s \in (t_1, t_2)$. Then, since $\mu_s$ is a unit density smooth flow around time $t$, each tangent flow at $X \in (\spt \mu)_t$ is a static multiplicity one plane. By Remark \ref{remark mult 1 static plane is fully smooth}, each $X \in (\spt \mu)_t$ is a fully smooth point so $(\sing^+ \hspace{-.04in} \mu)_t = \emptyset$. Since $t$ is chosen from a cocountable set, $(\sing^+ \hspace{-.04in} \mu)_t = \emptyset$ for all but countably many $t \in [0, T_{\disc})$. This proves the corollary under the first assumption.

\noindent \textbf{Brakke flow assumption:}

Under the assumption on the Brakke flow $\mu_t$, for all but countably many $t \in [0, T_{\disc})$, $\Te(x,t)\leq 1$ for each $x \in \bR^{n+k}$. Thus, if $(x,t) \in (\spt \mu)_t$, then $\Te(x,t) = 1$. Since $\mu_t$ is unit regular, each tangent flow at $(x,t)$, is a static multiplicity one planar tangent flow.  By Remark \ref{remark mult 1 static plane is fully smooth}, $(x,t)$ is a fully smooth point for $\mu_t$, and so $(\sing^+ \hspace{-.04in}\mu)_t = \emptyset$. This holds for all but countably many $t <T$, and we conclude the corollary.
\end{proof}

The next proposition says that the countability assumption on the set of singular times implies that $\mu_t$ satisfies Conjecture \ref{Conjecture C}. This will be used to prove Theorem \ref{theorem main assumption BFE before tau}.

\begin{prop}\label{proposition conjecture C holds for countable singular times}
Suppose that one of the following assumptions holds:
\begin{itemize}
    \item $(M_0, \mu_t)$ satisfies the General Assumption for $t<T$, the level set flow of $M_0$ has an open cocountable set of regular times for $t<T=T_{\disc}$, and $k=1$,
    \item $\mu_t$ is an integral $n$-dimensional Brakke flow in $\bR^{n+k}$ which is unit regular, has bounded area ratios, and has a countable set of singular times for $t<T$.
\end{itemize}

Then, for each $\phi \in C^2_c(\bR^{n+k}, \bR_{\geq 0})$ and all but countably many $t \in [0, T)$, the derivative $\frac{d}{dt}\mu_t(\phi)$ exists and
$$\frac{d}{dt}\mu_t(\phi) = \cB(\mu_t, \phi)$$
where $\cB(\mu_t, \phi)$ is the Brakke variation.
\end{prop}
\begin{proof}
By Corollary \ref{corollary main assumption implies vanishing singular set each time}, either assumption implies that for all but countably many $t \in [0, T)$, $(\sing^+ \hspace{-.04in}\mu)_t = \emptyset$.

Fix $\phi \in C^2_c(\bR^{n+k}, \bR_{\geq 0})$. For $t \in (0,T)$ such that $(\sing^+ \hspace{-.04in}\mu)_t = \emptyset$, there exists $\eps$ such that $\mu_s\lfloor \supp(\phi) = \cH^n \lfloor M_s$, where $M_s$ is a smooth flow in $\supp(\phi)$ defined for $s \in (t - \eps^2, t+\eps^2)$. This follows from the fact that $\supp(\phi)$ is compact and the fact that there is a spacetime ball around each $(x,t) \in (\spt \mu)_t$ such that $\mu_t$ is fully smooth in the ball. It is well-known that a smooth flow $M_s$ satisfies 
\begin{equation*}\label{equation equality in brakke flow derivative for prop}\frac{d}{ds}\cH^n \lfloor M_s(\phi) = \cB(\cH^n \lfloor M_s, \phi)\end{equation*}
By the fact that $\mu_s \lfloor \supp(\phi) = \cH^n \lfloor M_s$, we conclude that
$$\frac{d}{dt} \mu_t(\phi) = \cB(\mu_t, \phi)$$
where the derivative is evaluated at $t$. So, for all but countably many $t \in [0, T)$, $\frac{d}{dt}\mu_t(\phi)$ exists and $\frac{d}{dt}\mu_t(\phi) = \cB(\mu_t, \phi)$.
\end{proof}

\begin{prop}\label{proposition examples of main assumption}
Let $M_0 \subset \bR^{n+1}$ be a smooth closed embedded hypersurface. 
Suppose that one of the following assumptions holds:
\begin{itemize}
    \item The level set flow of $M_0$ has three-convex blow-up type and $T=T_{\disc}$,
    \item The level set flow of $M_0$ has only spherical singularities and neckpinches, and $T= \infty$.
\end{itemize}
Then, $(M_0, \mu_t)$ satisfies the General Assumption for $t<T$, and
\begin{enumerate}
    \item $\mu_t$ has a countable set of singular times for $t<T$,
    \item the level set flow of $M_0$ has an open cocountable set of regular times for $t<T$.
\end{enumerate}
\end{prop}
\begin{proof}\hfill\\
\textbf{Three-convex blow-up type:}

The level set flow of $M_0$ is nonfattening for $t< T = T_{\disc}$. By elliptic regularization (Theorem \ref{theorem unit density for nonfattening}), there is an associated unit density Brakke flow $\mu_t$ so that $(M_0, \mu_t)$ satisfies the General Assumption for $t< T$. Proposition \ref{proposition no quasistatic mult 1 planes} implies that $\mu_t$ is unit regular until time $T_{\disc}$.

By Lemma \ref{lemma lsf blowups are bf blowups} and Proposition \ref{proposition no Tmult for lsf with cylindrical singularities}, all tangent flows of $\mu_t$, until time $T$, are static multiplicity one planes or round shrinking cylinders $S^{n-k}\times \bR^k$ for $k = \{0,1,2\}$.
We also note that all Gaussian density ratios of $\mu_t$ are uniformly bounded since $M_0$ is smooth and embedded (see Theorem \ref{theorem huisken monotonicity}). 

We may now apply Colding-Minicozzi's stratification theory to $\mu_t$~\cite[Theorem 3.20, Theorem 0.3]{ColMin16}. Colding-Minicozzi's result applies for Brakke flows with only cylindrical tangent flows, i.e.\ any tangent flow with $\Te>1$ is a round shrinking cylinder. These Brakke flows are assumed to be unit regular, have uniform bounds on Gaussian density ratios, and have a smooth closed embedded initial condition.\footnote{It does not say explicitly what the assumptions are for the stratification theory, but it is evident that it uses the same assumptions as the uniqueness result for cylindrical tangent flows~\cite{ColMin15}, on which the stratification theory depends. This is generalized to arbitrary codimension in ~\cite{ColMin19}.} So, $\mu_t$ satisfies the requisite conditions to apply the results of \cite{ColMin16}. 

The stratification theory implies that for Brakke flows with three-convex singularities, there are countably many connected components of the spacetime singular set~\cite[Theorem 3.20]{ColMin16}, each of which is contained in a time-slice~\cite[Proof of Theorem 0.3]{ColMin16}.\footnote{Colding-Minicozzi do not write this statement exactly, but it follows from their work.} Here, the singular set consists of points with $\Te>1$. So, applying this theory to $\mu_t$, for all but countably many $t \in [0, T)$, $\Te(x,t) =1$ for all $(x,t) \in (\spt \mu)_t$. Since $\Te(x,t)<1$ for $(x,t)\notin (\spt \mu)_t$, there is a cocountable set $\cI \subseteq [0, T)$ such that for $t \in \cI$, for each $x \in \bR^{n+1}$, $\Te(x,t)\leq 1$. This means that $\mu_t$ has a countable set of singular times for $t<T$. 

Since $\mu_t$ is unit regular, has bounded area ratios (see Theorem \ref{theorem huisken monotonicity}), and has a countable set of singular times for $t<T$, $(\sing^+\hspace{-.04in}\mu)_t = \emptyset$ for $t \in \cI$ by Corollary \ref{corollary main assumption implies vanishing singular set each time}. Arguing as in the proof of Proposition \ref{proposition conjecture C holds for countable singular times}, this implies that for $t \in \cI$, there is $\eps>0$ so that $\mu_s = \cH^n \lfloor M_s$ for $s \in (t - \eps^2, t+\eps^2)$, where $M_s$ is a smooth flow defined for $s \in (t - \eps^2, t+ \eps^2)$. This follows from the fact that $(\spt \mu)_t$ is compact and the fact that for $t \in \cI$, each $(x,t) \in (\spt \mu)_t$ is a fully smooth point. Thus, we find that $(\spt \mu)_s = M_s\times \{s\}$ for each $s \in (t - \eps^2, t+\eps^2)$. By Proposition \ref{proposition support is level set flow}, $\spt \mu = \{u=0\}$ for $t<T_{\disc}$. Intersecting with $\{s\}$, we get that $(\spt \mu)_s = F_s(M_0)\times \{s\}$. So, $F_s(M_0) = M_s$ for $s \in (t - \eps^2, t+\eps^2)$. This implies $t \in \cI$ is a regular time of the level set flow of $M_0$. This holds for all $t \in \cI$, and in fact, this shows that each $t\in\cI$ is an interior point of the set of regular times for the level set flow. Since $\cI$ is cocountable, we conclude that the level set flow of $M_0$ has an open cocountable set of regular times. 

\noindent \textbf{Spherical singularities and neckpinches:}

Applying \cite{ChoiHaslhoferHersh18, ChoiHaslhoferHershWhite19}, a level set flow with only spherical singularities and neckpinches is nonfattening and $T_{\disc}=\infty$. There is an associated unit density Brakke flow $\mu_t$ from elliptic regularization, so $(M_0, \mu_t)$ satisfies the General Assumption for $t \in [0, \infty)$. The same argument as for three-convex blow-up type works to show that $\mu_t$ has a countable set of singular times for $t \in [0, \infty)$, and the level set flow of $M_0$ has an open cocountable set of regular times for $t \in [0, \infty)$.
\end{proof}

\section{Partial Regularity and Mass Continuity}\label{section partial regularity and mass continuity}
In this section, we prove Theorems \ref{theorem equivalence of mass continuity, codim 2} and \ref{theorem mass continuity implies mult one} relating a partial regularity condition to the no mass drop conjecture. Then, we prove Theorem \ref{theorem main result} with an application of stratification. Corollary \ref{corollary general assumption no mass drop} and Corollary \ref{corollary mean convex neighborhoods of singularities} follow, using Theorem \ref{theorem main result} and statements from Section \ref{section characterization of the main assumption}. We will then upgrade Theorem \ref{theorem equivalence of mass continuity, codim 2} to Theorem \ref{theorem main assumption BFE before tau} using the countability assumption on the set of singular times. Using Theorem \ref{theorem main assumption BFE before tau} and statements from Section \ref{section characterization of the main assumption}, we prove Corollary \ref{corollary two convex three convex BFE} and then Corollary \ref{corollary limit flows of BFEs}.

\subsection{Small Singular Sets and No Mass Drop}\hfill\\
\vspace{-.2in}

We first show that the a priori weaker assumption that $t \mapsto \bf{M}[\mu_t]$ is continuous implies that $t \mapsto \mu_t(\phi)$ is continuous. The point of this proposition is that for compact flows, continuity of $t \mapsto \bf{M}[\mu_t]$ is no less general than continuity of $t\mapsto \mu_t(\phi)$ for each $\phi \in C^2_c(\bR^{n+k}, \bR_{\geq 0})$. 
\begin{prop}\label{proposition mass continuous implies mass measure continuous}
Let $\{\mu_t\}_{t \geq 0}$ be an integral Brakke flow supported in a fixed ball for all time.\\
If $t \mapsto \bf{M}[\mu_t]$ is continuous, then $t \mapsto \mu_t(\phi)$ is continuous for each $\phi \in C^2_c(\bR^{n+k}, \bR_{\geq 0})$.
\end{prop}
\begin{proof} 
Let $B$ be the ball in which $\mu_t$ is supported for all time. This means there exists a smooth cutoff function $\psi \in C^2_c(\bR^{n+k}, \bR_{\geq 0})$ such that $\psi \equiv 1$ on $B$ and $\mu_t(\psi) = \bf{M}[\mu_t]$ for all $t \geq 0$.

Let $\phi \in C^2_c(\bR^{n+k}, \bR_{\geq 0})$ such that $\phi \leq 1$. Let $t_0 \geq 0$. Using the definition of $\psi$ and the fact that the right limits of $\mu_t$ exist by Lemma \ref{lemma properties of BF},
\begin{align*}
     \lim_{s \searrow t_0} \mu_s(\phi) + \lim_{s \searrow t_0} \mu_s(\psi(1-\phi)) &= \lim_{s \searrow t_0} \mu_s(\psi\phi) + \lim_{s \searrow t_0} \mu_s(\psi(1-\phi))\\
     &= \lim_{s \searrow t_0} \mu_s((\phi + 1-\phi)\psi)\\
     &= \lim_{s \searrow t_0} \bf{M}[\mu_s]\\
\intertext{Using the mass continuity assumption,}
     &= \bf{M}[\mu_{t_0}]\\
     &= \mu_{t_0}(\psi \phi)+ \mu_{t_0}(\psi (1-\phi))\\
     &= \mu_{t_0}(\phi) + \mu_{t_0}(\psi(1-\phi))
\end{align*}
By Lemma \ref{lemma properties of BF}, we have that 
$$\lim_{s \searrow t_0}\mu_s(\phi) \leq \mu_{t_0}(\phi)$$
$$\lim_{s \searrow t_0}\mu_s(\psi(1- \phi)) \leq \mu_{t_0}(\psi(1-\phi))$$
We conclude that 
\begin{equation}\label{equation one sided limit mass measure continuous}
\lim_{s \searrow t_0} \mu_s(\phi) = \mu_t(\phi)\end{equation}
for each $\phi\leq 1$. By scaling and the fact that $\phi$ was arbitrary, (\ref{equation one sided limit mass measure continuous}) holds for any $\phi \in C^2_c(\bR^{n+k}, \bR_{\geq 0})$. Also, $t_0$ was arbitrary, so (\ref{equation one sided limit mass measure continuous}) holds for any $t_0$.

Applying the same reasoning for $s \nearrow t_0$ and any $t_0 >0$, we conclude that $t \mapsto \mu_t(\phi)$ is continuous for each $t\geq 0$ and $\phi$.
\end{proof}

We now show that a partial regularity condition implies that the mass is continuous. Recall that $(\sing^+ \hspace{-.04in}\mu)_t$ is defined to be all the points in the time-$t$ slice of the spacetime support, $(\spt \mu)_t$, which are not fully smooth points (see the definitions of Section \ref{section singularities mean curvature flow}).

\begin{thm}\label{theorem equivalence of mass continuity, codim 2}
Let $T>0$. Let $\{\mu_t\}_{t\geq 0}$ be an integral $n$-dimensional Brakke flow in $\bR^{n+k}$ with bounded area ratios. Suppose that for each $t \in (0,T)$,
     \begin{equation}\label{equation codimension 1 assumption on singular set}
 \cH^{n}((\sing^+\hspace{-.04in} \mu)_t)=0
 \end{equation}
Then, for each $\phi \in C^2_c(\bR^{n+k}, \bR_{\geq 0})$, $t \mapsto \mu_t(\phi)$ is continuous for $t\in (0,T)$.
\end{thm}
\begin{rmk}\label{remark unit density immanent}
We stress that $(\sing^+\hspace{-.04in} \mu)_t$ includes any points with quasistatic planar tangent flows. A smooth flow which suddenly disappears will not satisfy (\ref{equation codimension 1 assumption on singular set}) at the vanishing time. We also note that a unit density assumption in Theorem \ref{theorem equivalence of mass continuity, codim 2} is unnecessary since (\ref{equation codimension 1 assumption on singular set}) has an inherent unit density assumption already. Fully smooth points are defined so that the flow is unit density in a spacetime neighborhood. 
\end{rmk}
\begin{proof}

Define 
$$S_t := (\sing^+\hspace{-.04in} \mu)_t$$




Fix $t_0 \in (0,T)$. 
In order to prove this lemma, we will show that for each $ \phi \in C^2_c(\bR^{n+k}, \bR_{\geq 0})$, $t_0$ is a continuity point for $t \mapsto \mu_t(\phi)$. Since $\mu_t$ has bounded area ratios, $\mu_t$ has locally bounded mass in particular and $\mu_t(\phi)$ is finite for each $\phi \in C^2_c(\bR^{n+k}, \bR_{\geq 0})$.

Let $\pi: \bR^{n+k}\times \bR_{\geq 0} \to \bR^{n+k}$ be the projection onto $\bR^{n+k}$, i.e.\ $\pi(x,t) = x$. We will often use the projections $\pi((\spt \mu)_{t_0})$ and $\pi(S_{t_0})$, since $(\spt \mu)_{t_0}$ and $S_{t_0}$ are subsets of the time slice $\{t=t_0\}$.

\begin{lem}\label{lemma convergence of functions away from singular set}
For $\psi \in C^2_c(\bR^{n+k}, \bR_{\geq 0})$ compactly supported in $\bR^{n+k} \setminus \pi(S_{t_0})$,
$$\lim_{s \to t_0} \mu_s(\psi) = \mu_{t_0}(\psi)$$
\end{lem}
\begin{proof}
Let $\psi_1\in C(K, \bR_{\geq 0})$ for a compact set $K \subset \bR^{n+k} \setminus \pi((\spt \mu)_{t_0})$. We have that
\begin{equation}\label{equation psi in W has vanishing limit mu}
\lim_{s \to t_0} \mu_s(\psi_1) = 0    
\end{equation}
 
If (\ref{equation psi in W has vanishing limit mu}) were not true, then there would exist a sequence of times $s_j \to t_0$ such that $\lim_{s_j \to t_0} \mu_{s_j}(\psi_1) >0$. So there would exist a sequence of points $y_j \in \spt \mu_{s_j}$ such that $y_j \in K$. Since $K$ is compact and disjoint from the closed set $\pi((\spt \mu)_{t_0})$, some subsequence of $(y_j,s_j)$ converges to $(x,t_0) \neq (\spt \mu)_{t_0}$. Since $\spt \mu$ is a closed subset of spacetime,
$(x,t_0) \in (\spt \mu)_{t_0}$, which is a contradiction. This proves (\ref{equation psi in W has vanishing limit mu}).
\hfill\\

By definition of $S_{t_0}$, 
each $(x,t_0)\in (\spt \mu)_{t_0} \setminus S_{t_0}$ is a fully smooth point. This means that there exists an open ball $B(x) \subset \bR^{n+k}$ and a time interval $(t_0 - \eps^2, t_0 + \eps^2)$ such that $\mu_s\lfloor B(x) = \cH^n \lfloor M_s$ for a smooth proper embedded mean curvature flow $M_s \subset B(x)$ with bounded curvature for $s \in (t_0 - \eps^2, t_0 + \eps^2)$.
For $x \in \pi((\spt \mu)_{t_0} \setminus S_{t_0})$, let $\phi_x \in C_c(B(x), \bR_{\geq 0})$ be compactly supported in $B(x)$.  Since $M_s$ is a smooth proper flow and such flows have continuous mass,
\begin{equation}\label{equation limit of flows}\lim_{s \to t_0} \mu_s(\phi_x) = \mu_{t_0}(\phi_x)\end{equation}
for any such $\phi_x$.

Now, define 
$$B':= \bR^{n+k} \setminus \pi((\spt \mu)_{t_0})$$
Let $\cB$ be the open cover of $\bR^{n+k} \setminus \pi(S_{t_0})$ given by $B'$ together with $B(x)$ for each $x \in \pi((\spt \mu)_{t_0}\setminus S_{t_0})$. That is,
$$\cB := \{B'\}\cup \bigcup_{x \in \pi((\spt \mu)_{t_0}\setminus S_{t_0})}\{B(x)\}$$

Let $X = \{\eta_i\}_{i \in \cI}$ be a partition of unity subordinate to $\cB$. For each compact set $K \subset \bR^{n+k} \setminus \pi(S_{t_0})$, all but finitely many functions in $X$ vanish in $K$. Then, for $\psi \in C^2_c(\bR^{n+k}, \bR_{\geq 0})$ such that $\psi$ is compactly supported in $\bR^{n+k}\setminus \pi(S_{t_0})$, there exists a finite subset $\cI_0 \subset \cI$, depending on $\psi$, such that
$$\mu_s(\psi) = \sum_{i \in \cI_0} \mu_s(\eta_i \psi)$$
and this sum is finite and well-defined regardless of $s$. Since $X$ is subordinate to $\cB$, each $\eta_i$ is compactly supported in some set in $\cB$. So, for each $i \in \cI_0$, $\text{supp}(\eta_i)\subset B'$ or there exists $x_j \in \pi((\spt \mu)_{t_0}\setminus S_{t_0})$ such that $\text{supp}(\eta_j) \subset B(x_j)$. Define
$$\cI_1 := \{i \in \cI_0 \,|\,\text{supp}(\eta_i) \subset B(x_j)\text{ for some }x_j \in \pi((\spt \mu)_{t_0}\setminus S_{t_0})\}$$
$$\cI_2 := \{i \in \cI_0 \,|\,\text{supp}(\eta_i) \subset B'\}\setminus \cI_1$$
Since $X$ is subordinate to $B$, we have that $\cI_0 = \cI_1 \sqcup \cI_2$. 

Now, for $\psi \in C^2_c(\bR^{n+k}, \bR_{\geq 0})$ compactly supported in $\bR^{n+k} \setminus \pi(S_{t_0})$,
\begin{align*}
    \lim_{s \to t_0} \mu_s(\psi)&= \lim_{s \to t_0}\sum_{i \in \cI_0} \mu_s(\eta_i \psi)\\
    &= \lim_{s \to t_0} \sum_{i \in \cI_1} \mu_s(\eta_i \psi) + \lim_{s \to t_0}\sum_{i \in \cI_2} \mu_s(\eta_i \psi)\\
\intertext{For $i \in \cI_2$, $\text{supp}(\eta_i \psi)$ is a compact subset of $B' = \bR^{n+k} \setminus \pi((\spt \mu)_{t_0})$. Applying (\ref{equation psi in W has vanishing limit mu}),}
    &=\lim_{s \to t_0} \sum_{i \in \cI_1} \mu_s(\eta_i \psi)\\
\intertext{Using finiteness of the sum,}
    &= \sum_{i \in \cI_1} \lim_{s \to t_0}\mu_s(\eta_i \psi)\\
\intertext{Since $\eta_i \psi$ is compactly supported in some $B(x_j)$ for each $i \in \cI_1$, we may apply (\ref{equation limit of flows}).}
    &= \sum_{i \in \cI_1} \mu_{t_0}(\eta_i \psi) \\
\intertext{Again, for $i \in \cI_2$, $\mu_{t_0}(\eta_i \psi)= 0$ since $\text{supp}(\eta_i \psi)$ is a compact subset of $\bR^{n+k}\setminus \pi((\spt \mu)_{t_0})$.}
    &= \sum_{i \in \cI_1} \mu_{t_0}(\eta_i \psi) + \sum_{i \in \cI_2} \mu_{t_0}(\eta_i \psi)\\
    &= \sum_{i \in \cI_0} \mu_{t_0}(\eta_i \psi)\\
    &= \mu_{t_0}(\psi)
\end{align*}
This concludes the lemma.
\end{proof}

Note that $S_{t_0}$ is closed. Indeed, let $\{(x_j,t_0)\}_{j \geq 1} \subset S_{t_0}$ be a converging sequence $(x_j,t_0) \to (x,t_0)$. Since $(\spt \mu)_{t_0}$ is closed, $(x,t_0) \in (\spt \mu)_{t_0}$. If $(x,t_0)$ is a fully smooth point, then there would exist a small neighborhood around $(x,t_0)$ in $(\spt \mu)_{t_0}$ consisting of fully smooth points. This would imply that some $(x_j,t_0)$ is fully smooth, which contradicts that $(x_j,t_0) \in S_{t_0}$.

Let $\phi \in C^2_c(\bR^{n+k}, \bR_{\geq 0})$. Define

$$S^{\phi}_{t_0}:= S_{t_0}\cap (\text{supp}(\phi)\times \{t_0\})$$

Since $S_{t_0}$ is closed, $S_{t_0}^{\phi}$ is compact. Since $S_{t_0}, S_{t_0}^{\phi} \subset \{t=t_0\}$, $\pi(S_{t_0})$ and $\pi(S_{t_0}^{\phi})$ are closed and compact, respectively, in $\bR^{n+k}$.
\vspace{.1in}

Let $\eps >0$. By the assumption that $\cH^{n}(S_{t_0})=0$, $\cH^n(\pi(S_{t_0})) = 0$ as well since $S_{t_0} \subset \{t = t_0\}$. Then, there exists $\de(\eps)>0$ and a countable collection of open sets $\{C_{\eps,j}\}_{j \geq 1}$ in $\bR^{n+k}$ such that
\begin{enumerate}
    \item $\{C_{\eps, j}\}_{j \geq 1}$ is an open cover for $\pi(S_{t_0}^{\phi})$,
    \item For each $j$, $\diam C_{\eps, j} \leq \de(\eps)$ and $C_{\eps, j}\cap \pi(S_{t_0}^{\phi}) \neq \emptyset$, 
    \item 
    \begin{equation}\label{equation hausdorff bound}\omega_{n}\sum_{j \geq 1} \Big(\frac{\diam C_{\eps, j}}{2} \Big)^{n}< \eps\end{equation}
\end{enumerate}
By definition of the Hausdorff measure, we may find $\de(\eps)$ such that $\de(\eps) \to 0$ as $\eps \to 0$. 

Define $$C_{\eps}:= \bigcup_{j \geq 1} C_{\eps, j}$$

Since $\bR^{n+k} \setminus C_{\eps}$ is closed and disjoint from the compact $\pi(S_{t_0}^{\phi})$, there exists $\eta>0$ such that $d(\bR^{n+k} \setminus C_{\eps}, \pi(S_{t_0}^{\phi}))>\eta$. Then, there exists a smooth cutoff function $0\leq \psi_{\eps}\leq 1$ such that $\psi_{\eps} \equiv 1$ on an open set containing $\pi(S_{t_0}^{\phi})$ and $\psi_{\eps} \equiv 0$ on $\bR^{n+k} \setminus C_{\eps}$.

Since $1-\psi_{\eps}$ vanishes in an open set around the compact $\pi(S_{t_0}^{\phi})$, $\supp(1-\psi_{\eps})$ does not intersect $\pi(S_{t_0}^{\phi})$. Since $\pi(S_{t_0}^{\phi}) = \pi(S_{t_0})\cap \supp(\phi)$, $\supp(\phi(1-\psi_{\eps}))$ does not intersect $\pi(S_{t_0})$. So, $\phi(1-\psi_{\eps})$ is compactly supported in $\bR^{n+k}\setminus \pi(S_{t_0})$. Using this fact, we apply Lemma \ref{lemma convergence of functions away from singular set} to find
\begin{equation}\label{equation application of 4.3}
    \lim_{s \to t_0} \mu_{s}(\phi(1-\psi_{\eps})) = \mu_{t_0}(\phi(1-\psi_{\eps}))
\end{equation}
By Lemma \ref{lemma properties of BF},
\begin{align*}
    \mu_{t_0}(\phi)&\leq \lim_{s \nearrow t_0}\mu_{s}(\phi)\\
    &= \lim_{s \nearrow t_0}\big( \mu_{s}(\phi \psi_{\eps})+ \mu_{s}(\phi(1-\psi_{\eps}))\big)\\
\intertext{Since the left limits exist by Lemma \ref{lemma properties of BF},}
    &= \lim_{s \nearrow t_0} \mu_{s}(\phi\psi_{\eps})+\lim_{s \nearrow t_0} \mu_{s}(\phi(1-\psi_{\eps}))\\
\intertext{By (\ref{equation application of 4.3}),}
    &= \lim_{s \nearrow t_0} \mu_{s}(\phi\psi_{\eps}) + \mu_{t_0}(\phi(1-\psi_{\eps}))\\
    &\leq \lim_{s \nearrow t_0} \mu_{s}(\phi\psi_{\eps}) + \mu_{t_0}(\phi)\\
    &\leq \max(\phi)\lim_{s \nearrow t_0} \mu_{s}(\psi_{\eps}) + \mu_{t_0}(\phi)\\
    &\leq \max(\phi)\lim_{s \nearrow t_0} \mu_{s}(C_{\eps}) + \mu_{t_0}(\phi)\\
    &\leq \max(\phi)\lim_{s \nearrow t_0} \sum_{j \geq 1} \mu_{s}(C_{\eps, j})  + \mu_{t_0}(\phi)\\
\intertext{Replacing each $C_{\eps, j}$ with a ball $B(\diam C_{\eps, j}/2)$ of radius $\diam C_{\eps, j}/2$ containing $C_{\eps, j}$,}
    &\leq \max(\phi)\lim_{s \nearrow t_0} \sum_{j \geq 1} \mu_{s}(B(\diam C_{\eps, j}/2)) + \mu_{t_0}(\phi)\\
\intertext{Applying the bounded area ratios assumption, denoting the bound by $\La$ as in Definition \ref{definition bounded area ratios},}
    &\leq \max(\phi)\omega_{n}\La \lim_{s \to t_0} \sum_{j \geq 1} \Big(\frac{\diam C_{\eps, j}}{2}\Big)^{n} + \mu_{t_0}(\phi)\\
\intertext{By (\ref{equation hausdorff bound}),}
    &\leq \max(\phi)\La \,\eps + \mu_{t_0}(\phi)
\end{align*}
Since $\eps>0$ is arbitrary, we get that
\begin{equation*}\mu_{t_0}(\phi) \leq \lim_{s \nearrow t_0}\mu_{s}(\phi) \leq \mu_{t_0}(\phi)\end{equation*}
We then conclude that
\begin{equation}\label{equation mass continuity one direction}
\lim_{s \nearrow t_0}\mu_{s}(\phi) = \mu_{t_0}(\phi)
\end{equation}
\hfill\\
Now, we will prove the other direction of (\ref{equation mass continuity one direction}), beginning with a lemma.
\begin{lem}\label{lemma liminf psieps}
$$\lim_{\eps \to 0}\mu_{t_0}(\phi(1-\psi_{\eps})) = \mu_{t_0}(\phi)$$
\end{lem}
\begin{proof}
We will first see that $\limsup_{\eps \to 0}\psi_{\eps}(x) = 0$ for almost every $x \in \bR^{n+k}$. 

Suppose that for $x_0 \in \bR^{n+k}$, $\limsup_{\eps \to 0}\psi_{\eps}(x_0) >0$. Since $\psi_{\eps} \equiv 0$ on $\bR^{n+k}\setminus C_{\eps}$, there is a sequence $\eps_i \to 0$ such that $x_0 \in C_{\eps_i}$ for each $i$. By construction, each $C_{\eps_i}$ is contained in the $\de(\eps_i)$-neighborhood of $\pi(S_{t_0}^{\phi})$. Recall that $\de(\eps_i) \to 0$ as $\eps_i \to 0$. Since $\pi(S_{t_0}^{\phi})$ is compact, $\pi(S_{t_0}^{\phi}) = \bigcap_{i \geq 1} C_{\eps_i}$. This means that $x_0 \in \pi(S_{t_0}^{\phi})$. Since $\pi(S_{t_0}^{\phi})$ has measure zero, $\limsup_{\eps \to 0}\psi_{\eps}(x) =0$ for almost every $x \in \bR^{n+k}$. 

Now, $\phi$ is integrable and dominates $\phi \psi_{\eps}$ for each $\eps$. Applying the reverse Fatou lemma and the fact that $\limsup_{\eps \to 0}\psi_{\eps}(x) = 0$ for a.e.\ $x \in \bR^{n+k}$,
$$\limsup_{\eps \to 0} \mu_{t_0}(\phi \psi_{\eps}) \leq \mu_{t_0}(\limsup_{\eps \to 0}\phi \psi_{\eps}) = 0$$
Since $\mu_{t_0}(\phi \psi_{\eps}) \geq 0$, $$\lim_{\eps \to 0} \mu_{t_0}(\phi \psi_{\eps})=0$$
Then,
\begin{align*}
    \mu_{t_0}(\phi) &= \mu_{t_0}(\phi) - \lim_{\eps\to 0}\mu_{t_0}(\phi \psi_{\eps})\\
    &= \lim_{\eps \to 0} \mu_{t_0}(\phi(1- \psi_{\eps}))
\end{align*}
\end{proof}

We may now finish the final calculation. Applying Lemma \ref{lemma properties of BF}, using that the right limits exist,
\begin{align*}
    \mu_{t_0}(\phi) & \geq \lim_{s \searrow t_0}\mu_s(\phi)\\
    & = \lim_{s \searrow t_0}\mu_s(\phi(1-\psi_{\eps})) + \lim_{s \searrow t_0}\mu_s(\phi\psi_{\eps})\\
    &\geq \lim_{s \searrow t_0}\mu_s(\phi(1-\psi_{\eps}))\\
\intertext{Applying Lemma \ref{lemma convergence of functions away from singular set}, since $\phi(1-\psi_{\eps})$ is compactly supported in $\bR^{n+k}\setminus \pi(S_{t_0})$,}    
    &= \mu_{t_0}(\phi(1 - \psi_{\eps}))
\end{align*}
Taking the limit of both sides as $\eps \to 0$ and using Lemma \ref{lemma liminf psieps},
$$\mu_{t_0}(\phi) \geq \lim_{s \searrow t_0}\mu_s(\phi) \geq \mu_{t_0}(\phi)$$
We conclude that
\begin{equation}\label{equation mass one direction}
    \lim_{s \searrow t_0} \mu_s(\phi) = \mu_{t_0}(\phi)
\end{equation}

Combining (\ref{equation mass continuity one direction}) and (\ref{equation mass one direction}), we find that $t_0$ is a continuity point for $t \mapsto \mu_{t}(\phi)$. This is true for each $t_0 >0$ and $\phi \in C^2_c(\bR^{n+k}, \bR_{\geq 0})$, so the theorem follows. 
\end{proof}

\begin{rmk}
The proof of Theorem \ref{theorem equivalence of mass continuity, codim 2} also works in any complete smooth ambient Riemannian manifold.
\end{rmk}

We will now prove the converse statement of Theorem \ref{theorem equivalence of mass continuity, codim 2}. This is based on the Brakke regularity theorem, which has particularly strong consequences assuming mass continuity.

\begin{thm}\label{theorem mass continuity implies mult one}
Let $T>0$. Let $\{\mu_t\}_{t\geq 0}$ be an integral $n$-dimensional Brakke flow in $\bR^{n+k}$ which is unit density and has bounded area ratios. Suppose that for each $\phi \in C^2_c(\bR^{n+k}, \bR_{\geq 0})$, $t \mapsto \mu_t(\phi)$ is continuous for $t \in (0,T)$.

Then, for each $t \in (0,T)$,
\begin{equation}\label{equation Hn zero set of mult planes}
 \cH^{n}((\sing^+\hspace{-.04in}\mu)_t)=0
\end{equation}
\end{thm}
\begin{proof}
Brakke showed that unit density $n$-dimensional Brakke flows with continuous mass have $\cH^{n}$-a.e.\ regularity of $\spt \mu_t$ for each time~\cite[Theorem 6.12]{Brakke1978}. That is, the assumption that $\mu_t$ is unit density and $t\mapsto \mu_t(\phi)$ is continuous for $t\in (0,T)$ implies that for each $t_0 \in (0,T)$, there is a closed set $C_{t_0} \subset \bR^{n+k}$, with $\cH^n(C_{t_0})=0$, such that for each $x \in \spt \mu_{t_0} \setminus C_{t_0}$, there is a ball $B_{\eps}(x)$ such that $\spt \mu_t$ is a smooth flow of embeddings in $B_{\eps}(x)$ for $t \in (t_0 - \eps^2, t_0 + \eps^2)$ (see also \cite[Theorem 9.7]{lahiri2014regularity}). That is, the support is smooth in a spacetime neighborhood of $(x,t_0)$, where $x \in \spt \mu_{t_0}\setminus C_{t_0}$. We may choose $\eps$ small enough so that $\spt \mu_t$ is connected in $B_{\eps}(x)$ for $t \in (t_0 - \eps^2, t_0 + \eps^2)$. Applying Lemma \ref{lemma restarting with unit density}, using the fact that $\mu_t$ is unit density and that $\spt \mu_t$ is smooth with bounded curvature in $B_{\eps}(x)$, we find that for $t \in (t_0 - \eps^2, t_0 + \eps^2)$,
\begin{equation}\label{equation static mult one plane on support for mass cont}\mu_t\lfloor B_{\eps}(x) = \cH^n \lfloor \spt \mu_t\cap B_{\eps}(x)\end{equation}
Since $\spt \mu_t$ is a smooth flow of embeddings in $B_{\eps}(x)$, (\ref{equation static mult one plane on support for mass cont}) implies that there is a static multiplicity one planar tangent flow at $(x,t_0)$. So, there is a static multiplicity one planar tangent flow at each $(x,t_0)$ with $x \in \spt \mu_{t_0}\setminus C_{t_0}$. So, each $(x,t_0)$, with $x \in \spt \mu_{t_0}\setminus C_{t_0}$, is a fully smooth point by Remark \ref{remark mult 1 static plane is fully smooth}. By the definition of $(\sing^+ \hspace{-.04in}\mu)_{t_0}$ and the fact that $\cH^n(C_{t_0})=0$,
\begin{equation}\label{equation tmult(t) first restriction}\cH^n(\pi((\sing^+\hspace{-.04in}\mu)_{t_0})\cap \spt \mu_{t_0}) = 0\end{equation}
for each $t_0 \in (0,T)$, where $\pi: \bR^{n+k}\times \bR_{\geq 0} \to \bR^{n+k}$ is the projection function onto $\bR^{n+k}$ in spacetime, i.e.\ $\pi(x,t) = x$. 
By definition, $(\sing^+\hspace{-.04in}\mu)_{t_0}\subset (\spt \mu)_{t_0}$, so 
\begin{equation}\label{equation tmult(t) second restriction}
    \cH^n\big(\pi((\sing^+\hspace{-.04in}\mu)_{t_0})\cap (\bR^{n+k}\setminus \pi((\spt \mu)_{t_0}))\big) = 0
\end{equation}
Now, define
$$G(t_0) := \pi((\spt \mu)_{t_0})\setminus \spt \mu_{t_0}$$
The ultimate goal is to show that $\cH^n(\pi((\sing^+\hspace{-.04in}\mu)_{t_0}))=0$, which implies the desired result (\ref{equation Hn zero set of mult planes}). Given (\ref{equation tmult(t) first restriction}) and (\ref{equation tmult(t) second restriction}), we only need to show 
\begin{equation}\label{equation intermediate sing}\cH^n\big(\pi((\sing^+\hspace{-.04in}\mu)_{t_0})\cap G(t_0)\big)=0\end{equation}
We will prove (\ref{equation intermediate sing}) by proving that $\cH^n(G(t_0))=0$ in Lemma \ref{lemma lahiri's support lemma}. This lemma seems to be a corollary of Lahiri's proof of the Brakke regularity theorem \cite[Theorem 9.7]{lahiri2014regularity}, but we provide a proof here for completeness. Our proof of Lemma \ref{lemma lahiri's support lemma} is morally speaking similar to the approach of \cite[Theorem 9.7]{lahiri2014regularity}, but we use a more standard, elementary version of Brakke's clearing out lemma.

\begin{lem}\label{lemma lahiri's support lemma}
For each $t_0 >0$,
$$\cH^n(G(t_0)) = 0$$
\end{lem}
\begin{proof}
We will use the following lemma, due to Lahiri~\cite[Lemma 9.5]{lahiri2014regularity}, which is largely a modification of Brakke's argument in \cite[Theorem 6.12]{Brakke1978}. The idea of this lemma is that the set of points which have microscopic drops in area ratio bigger than the threshold $\tau>0$ has $\cH^n$ measure zero.

\begin{lem}[{\cite[Lem.\ 9.5]{lahiri2014regularity}}]\label{lemma lahiri's lemma}
For all $R, L, \tau \in (0,\infty)$, the following holds: Let $\mu_t$ be an integral Brakke flow in $B_{2R}(0)$ defined for $t \in [-R^2, R^2]$ such that $t \mapsto \mu_t(\phi)$ is continuous at $t=0$ for each $\phi \in C^2_c(B_{2R}(0), \bR_{\geq 0})$. For $\de \in (0,R)$, let
\begin{align*}
    D_{\mu}(\tau, \de) &:= \{x \in B_R(0)\,|\,\cD(x,\de)\geq \tau\}\\
    \cD_{\mu}(x, \de) &:= \sup_{\phi} \sup_{t \in (-\de^2, \de^2)} |\de^{-n}\mu_t(\phi) - \de^{-n}\mu_0(\phi)|
\end{align*}
where the supremum is taken over $\phi \in C_c^{0,1}(B_{\de}(x), [0,1])$ such that $\Lip(\phi) \leq \de^{-1}L$.

Then, 
$$\cH^n\Big(\hspace{-.03in}\bigcap_{\de \in (0,R)} D_{\mu}(\tau, \de)\Big) = 0$$
\end{lem}

Let $x \in G(t_0)$. Since $x \notin \spt \mu_{t_0}$ and $\spt \mu_{t_0}$ is closed, there is $r:= r(x)$ and a ball $B_r(x)$ such that $B_r(x) \subset \bR^{n+k}\setminus \spt \mu_{t_0}$. Then, $\mu_{t_0}(\phi) = 0$ for each $\phi \in C^2_c(B_r(x), \bR_{\geq 0})$. 
\begin{lem}\label{lemma lahiri D lower bound}
For $t_0 >0$, let $x \in G(t_0)$ and let $\tilde{\mu}_t$ be the time-shifted Brakke flow $\tilde{\mu}_t:= \mu_{t+t_0}$. Then, there exists $\tau = \tau(n,k)$ such that for $L=L(n,k)$ and each $\de \in (0,r(x))$, 
$$\cD_{\tilde{\mu}}(x,\de)\geq \tau$$
\end{lem}
\begin{proof}
Recall Brakke's clearing out lemma~\cite[6.3]{Brakke1978} (see~\cite[12.2]{Ilmanen94} for this particular statement): there are constants $\eta>0$ and $0<c_1<c_2$ depending only on $n$ and $k$ such that for any integral Brakke flow $\{\mu_t\}_{t \geq 0}$, $\de>0$, and $(x,t_0) \in \bR^{n+k}\times \bR_{\geq 0}$, if $\mu_{t_0}(B_{\de}(x))\leq \eta\, \de^n$, then $\mu_t(B_{\de/2}(x)) = 0$ for $t \in [t_0 + c_1 \de^2, t_0 + c_2 \de^2]$.

Let $x \in G(t_0)$. For $\de < r(x)$, $B_{\de}(x) \subset \bR^{n+k}\setminus \spt \mu_{t_0}$, so $\tilde{\mu}_0(B_{\de}(x))=0$. Thus, for $\de \in (0,r(x))$ and each $L$,
\begin{equation}\label{equation reduction to support}\cD_{\tilde{\mu}}(x,\de) = \sup_{\phi} \sup_{t \in (-\de^2, \de^2)} \de^{-n}\tilde{\mu}_t(\phi)\end{equation}
where the supremum is taken over $\phi \in C^{0,1}_c(B_{\de}(x), [0,1])$ such that $\Lip(\phi) \leq \de^{-1}L$.

Suppose that for constants $C$ and $L$ to be chosen later,
\begin{align}\label{equation supposition}
    \cD_{\tilde{\mu}}(x, \de_0) < C \eta
\end{align}
for some $\de_0 \in (0,r(x))$ where $\eta$ is as in the clearing out lemma.

Let $\phi \in C_c^{2}(B_1(0), [0,1])$ such that $\phi \equiv 1$ on $B_c(0)$ for $c \in (0,1)$ to be chosen later and $\Lip(\phi) \leq L= 10(1-c)^{-1}$. Then, define $\phi_{\de_0}(\cdot) = \phi(\de_0^{-1}(\cdot - x))$. So, $\phi_{\de_0} \in C^2_c(B_{\de_0}(x), [0,1])$ with $\Lip(\phi_{\de_0}) \leq \de_0^{-1}L$.

Now, (\ref{equation reduction to support}) and (\ref{equation supposition}) imply that for $\de_0$ and $L$,
\begin{align*}C\eta> \cD_{\tilde{\mu}}(x,\de_0) &= \sup_{\phi} \sup_{t \in (-\de_0^2, \de_0^2)} \de_0^{-n}\tilde{\mu}_t(\phi)\\
&\geq \sup_{t \in (-\de_0^2, \de_0^2)} \de_0^{-n}\tilde{\mu}_t(\phi_{\de_0})\\
&\geq \sup_{t \in (-\de_0^2, \de_0^2)} \de_0^{-n} \tilde{\mu}_t(B_{c\de_0}(x))
\end{align*}
Thus, for $C = c^n$,
\begin{equation}\label{equation clearing out initial}\sup_{t \in (-\de_0^2, \de_0^2)} \tilde{\mu}_t(B_{c\de_0}(x)) < \eta (c\de_0)^n\end{equation}
Let $c_1, c_2$ be the constants from the clearing out lemma. If $c_1 > 1$, choose $c \in (0,1)$ such that $c_1 < c^{-2} < c_2$. If $c_1 \leq 1$, let $c = 1/2$. Then, choose $t_0 \in (-\de_0^2, 0)$ such that $|t_0| \in (c_1(c\de_0)^2, c_2 (c\de_0)^2)$. This choice of $t_0$ is possible by our choice of $c$. These choices imply that
$$0 \in I:= [t_0 + c_1(c\de_0)^2, t_0 + c_2 (c\de_0)^2]$$
By (\ref{equation clearing out initial}), $\tilde{\mu}_{t_0}(B_{c\de_0}(x)) < \eta (c\de_0)^n$. Applying the clearing out lemma to $I$, $$\tilde{\mu}_t(B_{c\de_0/2}(x)) = 0$$ 
for each $t \in I$. This means that $\tilde{\mu}_t\equiv 0$ in a spacetime neighborhood of $(x,0)$. So, $\mu_t \equiv 0$ in a spacetime neighborhood of $(x,t_0)$. This implies that $(x,t_0)$ does not belong to the spacetime support of $\mu$ and in particular, $(x,t_0) \notin (\spt \mu)_{t_0}$. This means that $x \notin \pi((\spt \mu)_{t_0})$.

Thus, if $x \in G(t_0)$ and we choose $c$ depending only on $c_1, c_2$ as above, choose $C = c^n$, choose $L = 10(1-c)^{-1}$, and if $\cD_{\tilde{\mu}}(x,\de_0)< C\eta$ for some $\de_0 \in (0,r(x))$, then $x \notin \pi((\spt \mu)_{t_0})$. Since $x\in G(t_0) \subset \pi((\spt \mu)_{t_0})$, this is a contradiction. Note that all choices of constants depend on the constants $c_1, c_2, \eta$ from the clearing out lemma, which only depends on $n$ and $k$. So, there exist constants $\tau = C\eta$ and $L$ depending only on $n$ and $k$, such that for $x \in G(t_0)$, this choice of $L$, and each $\de \in (0,r(x))$,
$$\cD_{\tilde{\mu}}(x,\de)\geq \tau$$
This completes the proof of Lemma \ref{lemma lahiri D lower bound}.
\end{proof}
\noindent Now, for each positive integer $m \in \bZ_{>0}$, define $R_m := \min(1/m, \sqrt{t_0})$. Define
\begin{align*}
U_m &:= \{x \in \bR^{n+k}\,|\,d(x, \spt \mu_{t_0})> R_m\}\\
G_m &:= G(t_0) \cap U_m
\end{align*}

\noindent Consider the countable open cover of $U_m$
$$\cB_m:= \{B_{R_m}(y_j)\}_{j=1}^{\infty}$$
where $\{y_j\}_{j=1}^{\infty}$ is an enumeration of $\bQ \cap U_m$. Indeed, since $U_m$ is an open set, for each $y \in U_m$, there exists a subsequence $y_{j'} \in \bQ \cap U_m$ converging to $y$ and so $y \in B_{R_m}(y_{j'})$ for some $y_{j'}$. Thus, $\cB_m$ is a cover of $U_m$. In particular, $\cB_m$ covers $G_m$.

By definition, for each $x \in G_m$, $r(x) > 1/m$. This means that $R_m < r(x)$ for each $x \in G_m$. By Lemma \ref{lemma lahiri D lower bound}, for each $x \in G_m$, there exists $\tau$ and $L$ depending only on $n$ and $k$ such that for each $\de \in (0,R_m)$,
\begin{equation}\label{equation final lahiri D lower bound}\cD_{\tilde{\mu}}(x,\de) \geq \tau\end{equation}
where $\tilde{\mu}_t := \mu_{t+t_0}$. As in Lemma \ref{lemma lahiri's lemma}, define
$$D^{y_j, m}_{\tilde{\mu}}(\tau, \de) := \{x \in B_{R_m}(y_j)\,|\,\cD_{\tilde{\mu}}(x,\de) \geq \tau\}$$
By (\ref{equation final lahiri D lower bound}), for each $j$, we have that
\begin{equation}\label{equation subset lahiri D}
G_m \cap B_{R_m}(y_j) \subseteq \bigcap_{\de \in (0,R_m)} D^{y_j, m}_{\tilde{\mu}}(\tau, \de)\end{equation}

Although Lemma \ref{lemma lahiri's lemma} is written for flows defined in $B_{2R}(0)$, it applies just as well for balls centered on $y_j$. We apply Lemma \ref{lemma lahiri's lemma} to $\tilde{\mu}_t$ around $B_{2R_m}(y_j)$ using the fact that $t\mapsto \mu_t(\phi)$ is continuous.  Lemma \ref{lemma lahiri's lemma} says that
\begin{equation*} \cH^n\Big( \hspace{-.03in}\bigcap_{\de \in (0,R_m)} D^{y_j, m}_{\tilde{\mu}}(\tau, \de)\Big) = 0\end{equation*}
and so by (\ref{equation subset lahiri D}),
\begin{equation}\label{equation zero measure atom}
\cH^n(G_m \cap B_{R_m}(y_j)) = 0
\end{equation}
Since $\cB_m$ covers $G_m$, (\ref{equation zero measure atom}) implies by countable subadditivity that
\begin{equation}\label{equation Gm sum}\cH^n(G_m) = \cH^n\Big(\hspace{-.03in}\bigcup_{j=1}^{\infty} G_m \cap B_{R_m}(y_j)\Big) = 0\end{equation}
Since $r(x)>0$ for each $x \in G(t_0)$, there exists an $m$ such that $x \in G_m$ for each $x \in G(t_0)$. Thus,
\begin{equation}\label{equation Gm sum 2}G(t_0) = \bigcup_{m=1}^{\infty}G_m\end{equation}
So, (\ref{equation Gm sum}) and (\ref{equation Gm sum 2}) imply by countable subadditivity that
$$\cH^n(G(t_0)) = \cH^n\Big(\hspace{-.03in}\bigcup_{j=1}^{\infty}G_m\Big) = 0$$
This concludes Lemma \ref{lemma lahiri's support lemma}.
\end{proof}
Using the definition of $G(t_0)$, Lemma \ref{lemma lahiri's support lemma} tells us that (\ref{equation intermediate sing}) holds, i.e.\
\begin{equation}\label{equation tmult third restriction}\cH^n\big(\pi((\sing^+ \hspace{-.04in}\mu)_{t_0})\cap \big(\pi((\spt \mu)_{t_0})\setminus \spt \mu_{t_0}\big)\big) = 0\end{equation}
Combining (\ref{equation tmult(t) first restriction}), (\ref{equation tmult(t) second restriction}), and (\ref{equation tmult third restriction}),
$$\cH^n(\pi((\sing^+ \hspace{-.04in}\mu)_{t_0}))=0$$
for each $t_0>0$. Since $(\sing^+ \hspace{-.04in}\mu)_{t_0}$ is contained in the time slice $\{t=t_0\}$,
$$\cH^n((\sing^+\hspace{-.04in}\mu)_{t_0}) = \cH^n(\pi((\sing^+ \hspace{-.04in}\mu)_{t_0})) = 0$$
which concludes the theorem.
\end{proof}

Let $\{\mu_t\}_{t \geq 0}$ be an integral Brakke flow with bounded area ratios. 
Recall that for the Brakke flow $\mu_t$ and each $t>0$, we define
$$\cT_{\mult}(t) := \{x \in \bR^{n+k}\,|\,\exists \,\text{a planar tangent flow at }(x,t)\, \text{with multiplicity $\geq 2$}\}$$
where the tangent flows are of the flow $\mu_t$. 
For the Brakke flow $\mu_t$ and each $t>0$, we define
\begin{align*}\label{equation Tmult(t) definition}
\cT_{\mult}^+(t) := \cT_{\mult}(t) \cup \{x \in \bR^{n+k}\,|\,\exists \,\text{a quasistatic planar tangent flow at }(x,t)\}
\end{align*}
We are defining $\cT_{\mult}^+(t)$ to be $\cT_{\mult}(t)$ together with each $x \in \bR^{n+k}$ such that there is a tangent flow at $(x,t)$ which is a quasistatic plane. Since $\cT_{\mult}(t)$ already contains all $x$ such that there is a tangent flow at $(x,t)$ which is a quasistatic plane of multiplicity $\geq 2$, if $x \in \cT_{\mult}^+(t)\setminus \cT_{\mult}(t)$, then a tangent flow at $(x,t)$ is a quasistatic multiplicity one plane.

In order to relate the structure of tangent flows to partial regularity, we will use a stratification theorem of Brian White~\cite[Theorem 9, Table 2]{Wh97} (cf. \cite{CheegerHaslhoferNaber13}). The point is that for each time slice, we can bound the Hausdorff dimension of all points which only have nonplanar tangent flows. A planar tangent flow is any tangent flow which is a static or quasistatic plane of any positive integer multiplicity. White's stratification theorem implies the next theorem, Theorem \ref{theorem white stratification}.

\begin{thm}[{\cite[Thm.\ 9]{Wh97}}]\label{theorem white stratification}
Let $\{\mu_t\}_{t \geq 0}$ be an integral $n$-dimensional Brakke flow in $\bR^{n+k}$ with bounded area ratios. For $t>0$, define 
$$\fM(t):= \{x \in \bR^{n+k}\,|\,\text{no tangent flow at }(x,t)\,\text{is planar}\}$$
Then, for each $t>0$, the Hausdorff dimension of $\fM(t)$ is less than or equal to $n-1$.
\end{thm}

Using Theorem \ref{theorem white stratification}, we put together all the work of this section in the next theorem, Theorem \ref{theorem main result without unit regularity}, which is the antecedent to Theorem \ref{theorem main result}.

\begin{thm}\label{theorem main result without unit regularity}
Let $T>0$, and let $\{\mu_t\}_{t \geq 0}$ be an integral $n$-dimensional Brakke flow in $\bR^{n+k}$ with bounded area ratios. Consider the following conditions:
\begin{enumerate}
    \item\label{condition 1'} For each $t \in (0,T)$, $\cH^n(\cT^+_{\mult}(t))=0$,
    \item\label{condition 2'} For each $t \in (0,T)$, $\cH^n((\sing^+\hspace{-.04in}\mu)_t)=0$,
    \item\label{condition 3'} For each $\phi \in C^2_c(\bR^{n+k}, \bR_{\geq 0})$, $t \mapsto \mu_t(\phi)$ is continuous for $t \in (0,T)$.
\end{enumerate}
Then, (\ref{condition 1'}) is equivalent to (\ref{condition 2'}), and (\ref{condition 2'}) implies (\ref{condition 3'}). 

In addition, if $\mu_t$ is unit density, then conditions (\ref{condition 1'}), (\ref{condition 2'}), and (\ref{condition 3'}) are all equivalent.
\end{thm}
\begin{proof}
Theorem \ref{theorem equivalence of mass continuity, codim 2} says that condition (\ref{condition 2'}) implies condition (\ref{condition 3'}). We will now see that conditions (\ref{condition 1'}) and (\ref{condition 2'}) are equivalent. 

Let $\fM(t)$ be as in Theorem \ref{theorem white stratification}. Recall that $(\sing^+\hspace{-.04in}\mu)_t$ consists of all $(x,t)\in (\spt \mu)_t$ such that $(x,t)$ is not fully smooth. By Remark \ref{remark mult 1 static plane is fully smooth} and uniqueness of static multiplicity one planar tangent flows, $(\sing^+\hspace{-.04in}\mu)_t$ consists of all $(x,t)\in (\spt \mu)_t$ such that there is a tangent flow at $(x,t)$ which is not a static multiplicity one plane. By definition, 
$$\fM(t) \sqcup \cT_{\mult}^+(t) = \pi((\sing^+ \hspace{-.04in}\mu)_t)$$
where $\pi(x,t)=x$. Since $(\sing^+\hspace{-.04in}\mu)_t$ is contained in $\bR^{n+k}\times \{t\}$, $\cH^n((\sing^+\hspace{-.04in}\mu)_t) = \cH^n(\pi((\sing^+ \hspace{-.04in}\mu)_t))$. Using the fact that $\fM(t)$ and $\cT^+_{\mult}(t)$ are disjoint,
$$\cH^n(\fM(t)) + \cH^n(\cT_{\mult}^+(t)) = \cH^n((\sing^+ \hspace{-.04in}\mu)_t)$$
By Theorem \ref{theorem white stratification}, $\cH^n(\fM(t)) = 0$ for each $t>0$, so 
$$\cH^n(\cT_{\mult}^+(t)) = \cH^n((\sing^+ \hspace{-.04in}\mu)_t)$$
Thus, conditions (\ref{condition 1'}) and (\ref{condition 2'}) are equivalent.

Now, if $\mu_t$ is unit density, then Theorem \ref{theorem mass continuity implies mult one} says that (\ref{condition 3'}) implies (\ref{condition 2'}). Given that (\ref{condition 1'}) is equivalent to (\ref{condition 2'}), we find that (\ref{condition 1'}), (\ref{condition 2'}), and (\ref{condition 3'}) are all equivalent.
\end{proof}

\begin{rmk}\label{remark unit density unnecessary}
Conditions (\ref{condition 1'}) and (\ref{condition 2'}) of Theorem \ref{theorem main result without unit regularity} have built-in unit density assumptions. Both these conditions are equivalent to the statement that $\cH^n$-a.e.\ point of $(\spt \mu)_t$ has a static multiplicity one plane as a tangent flow. By Remark \ref{remark mult 1 static plane is fully smooth}, the flow is, in particular, unit density in a neighborhood of $\cH^n$-a.e.\ point of $\spt \mu_t$. So, this means that $\mu_t$ is unit density.
\end{rmk}

\subsection{Proof of Theorem \ref{theorem main result}}\hfill\\
\vspace{-.2in}

By Theorems \ref{theorem equivalence of mass continuity, codim 2} and \ref{theorem mass continuity implies mult one}, conditions (\ref{condition 2}) and (\ref{condition 3}) are equivalent. So, we only need to show that conditions (\ref{condition 1}) and (\ref{condition 2}) are equivalent.

By the discussion before Theorem \ref{theorem white stratification}, if $x \in \cT^+_{\mult}(t)\setminus \cT_{\mult}(t)$, then there is a quasistatic multiplicity one tangent flow of $\mu_t$ at $(x,t)$. Since $\mu_t$ is unit regular, there are no quasistatic multiplicity one tangent flows and so 
$$\cT^+_{\mult}(t) = \cT_{\mult}(t)$$
In particular, $\cH^n(\cT^+_{\mult}(t)) = 0$ if and only if $\cH^n(\cT_{\mult}(t)) = 0$. By Theorem \ref{theorem main result without unit regularity}, $\cH^n(\cT^+_{\mult}(t)) = 0$ if and only if $\cH^n((\sing^+ \hspace{-.04in}\mu)_t) = 0$, so condition (\ref{condition 1}) is equivalent to condition (\ref{condition 2}). This completes the proof of Theorem \ref{theorem main result}.

\subsection{Proof of Corollary \ref{corollary general assumption no mass drop}}\hfill\\
\vspace{-.2in}

We assume that $(M_0, \mu_t)$ satisfies the General Assumption for 
$$t< T:= \min(T_{\disc}, T_{\mult})$$

By Theorem \ref{theorem huisken monotonicity}, $\mu_t$ has bounded area ratios, since it has a smooth closed embedded connected initial condition $M_0$. Also, $\mu_t$ is a boundary motion and is unit density for each time.

By Proposition \ref{proposition no quasistatic mult 1 planes}, there are no quasistatic multiplicity one planar tangent flows of $\mu_t$ at each spacetime point $X = (x,t)$, $t \in (0,T)$. In other words, $\mu_t$ is unit regular. Also, by definition of $T_{\mult}$, there are no static or quasistatic planar tangent flows of multiplicity $\geq 2$ at $X$. So, for each $t \in (0, T)$,
$$\cH^n(\cT_{\mult}(t))=0$$

Thus, $\mu_t$ is unit density, is unit regular, has bounded area ratios, and satisfies condition (1) of Theorem \ref{theorem main result}. Theorem \ref{theorem main result} implies that $t \mapsto \mu_t(\phi)$ is continuous for $t\in (0, T)$ and each $\phi \in C^2_c(\bR^{n+1}, \bR_{\geq 0})$. Since $M_0$ is smooth, $\spt \mu_t$ will coincide with the smooth level set flow for short time. By Lemma \ref{lemma restarting with unit density}, $\mu_t$ will be a unit density smooth flow for short time. This implies that $t \mapsto \mu_t(\phi)$ is also continuous at $t=0$.


\subsection{Proof of Corollary \ref{corollary mean convex neighborhoods of singularities}}\hfill\\
\vspace{-.2in}

\noindent\textbf{Mean convex neighborhoods:}

Assume that the level set flow of $M_0$ has mean convex neighborhoods of singularities. By work of Hershkovits-White~\cite{HershkovitsWhite17}, $T_{\disc} = \infty$. Since it is nonfattening, we may associate the unit density Brakke flow $\mu_t$ to this level set flow. White's regularity theory of mean convex flows applies locally~\cite{Wh00}~\cite[Remark 5]{HershkovitsWhite17}, so there are no tangent flows at $X = (x,t)$, $t>0$, which are planes of multiplicity $\geq 2$ (see Proposition \ref{proposition no Tmult for lsf with cylindrical singularities}). Nor are there quasistatic multiplicity one planes as tangent flows by Proposition \ref{proposition no quasistatic mult 1 planes}. So, $T_{\disc} = T_{\mult} = \infty$.

We apply Corollary \ref{corollary general assumption no mass drop} to find that $t \mapsto \mu_t(\phi)$ is continuous for $t \in [0, \infty)$. 

\noindent\textbf{Generic flow:}

Assume that $M_0$ is a generic closed hypersurface. By Proposition \ref{proposition general assumption is generic}, $T_{\disc} = T_{\fat} = \infty$. We then apply Corollary \ref{corollary general assumption no mass drop}, as above, until time $T = T_{\mult}$.

\subsection{Integral Brakke Equality}\label{section applications of converging BFEs}\hfill\\
\vspace{-.2in}

We will now upgrade the mass continuity result of Theorem \ref{theorem main result} and Corollary \ref{corollary general assumption no mass drop} with countability assumptions on the set of singular times. We show that the mass is absolutely continuous with this assumption on the singular times. This will lead to the proofs of Theorem \ref{theorem main assumption BFE before tau} and Corollary \ref{corollary two convex three convex BFE}.

\begin{lem}\label{lemma absolutely continuous mass}
Suppose that one of the following assumptions holds:
\begin{itemize}
    \item $(M_0, \mu_t)$ satisfies the General Assumption for $t<T= \min(T_{\disc}, T_{\mult})$, the level set flow of $M_0$ has an open cocountable set of regular times for $t<T$, and $k=1$,
  \item $\mu_t$ is an integral $n$-dimensional Brakke flow in $\bR^{n+k}$ which is unit regular, has bounded area ratios, and has a countable set of singular times for $t<T= T_{\mult}$.
\end{itemize}
Then, for each $\phi \in C^2_c(\bR^{n+k}, \bR_{\geq 0})$, $t\mapsto \mu_t(\phi)$ is absolutely continuous for $t \in (0, T)$.
\end{lem}
\begin{proof} 
Fix $\phi \in C^2_c(\bR^{n+k}, \bR_{\geq 0})$. We will verify that $t \mapsto \mu_t(\phi)$ is absolutely continuous using the Banach-Zaretskii theorem~\cite[p. 232]{Nat74}. This theorem says that a function $f: [a,b] \to \bR$ is absolutely continuous if and only if $f$ is continuous, has bounded variation, and maps measure zero sets to measure zero sets. The last condition is also known as the Luzin N property. 

We will show that $t \mapsto \mu_t(\phi)$ satisfies the conditions of the Banach-Zaretskii theorem for each interval $[\eps,T - \eps]$, $\eps >0$. 

\noindent \textbf{Continuity:}

Under the level set flow assumption, Corollary \ref{corollary general assumption no mass drop} is satisfied, so $t \mapsto \mu_t(\phi)$ is continuous for $t \in (0, T)$. 

By the definition of $T_{\mult}$ and the fact that $\mu_t$ is unit regular, we have that $\cH^n(\cT_{\mult}^+(t)) = 0$ for each $t \in (0,T)$. Under the Brakke flow assumption, the assumptions of Theorem \ref{theorem main result without unit regularity} are satisfied. So, $t \mapsto \mu_t(\phi)$ is continuous for $t \in (0, T)$. 

Under either assumption, we have in particular that $t \mapsto \mu_t(\phi)$ is continuous on $[\eps,T - \eps]$.

\noindent \textbf{Bounded variation:}

By Lemma \ref{lemma properties of BF}, $\mu_t(\phi) - C(\phi)t$ is nonincreasing. Any flow with bounded area ratios has locally bounded mass, so $\mu_t$ has locally bounded mass. This implies $\mu_t(\phi) - C(\phi)t$ is bounded on $[0, T)$. Also, $C(\phi)t$ is increasing and bounded on $[0, T)$. Thus, we may represent $t \mapsto \mu_t(\phi)$ as the difference of bounded non-decreasing functions on $[\eps, T - \eps]$. By the Jordan decomposition for bounded variation functions, $t \mapsto \mu_t(\phi)$ has bounded variation on $[\eps, T - \eps]$.

\noindent \textbf{Luzin N property:}

A function $f: [a,b] \to \bR$ satisfies the Luzin N property if for all $N \subseteq [a,b]$ with $\cH^1(N) = 0$, we have that $\cH^1(f(N)) = 0$. By \cite[Lemma 7.25]{Rud87}, differentiable functions have this property. Functions differentiable at all but countably many times have this property also, since countable sets are always mapped to countable sets. By Proposition \ref{proposition conjecture C holds for countable singular times}, the assumptions of this lemma imply that $t \mapsto \mu_t(\phi)$ is differentiable for all but countably many times in $[0, T)$. So, $t \mapsto \mu_t(\phi)$ satisfies the Luzin N property on $[\eps, T - \eps]$.

We have verified all the conditions needed to apply the Banach-Zaretskii theorem, so $t \mapsto \mu_t(\phi)$ is absolutely continuous for $t \in [\eps, T - \eps]$. Since $\eps>0$ was arbitrary, we conclude the lemma. 
\end{proof}
\subsection{Proof of Theorem \ref{theorem main assumption BFE before tau}}\hfill\\
\vspace{-.2in}

\noindent\textbf{Level set flow assumption:}

Assume that $(M_0, \mu_t)$ satisfies the General Assumption for $t<T:= \min(T_{\disc}, T_{\mult})$ and the level set flow of $M_0$ has an open cocountable set of regular times.

By Lemma \ref{lemma absolutely continuous mass}, $t \mapsto \mu_t(\phi)$ is absolutely continuous for each $\phi \in C^2_c(\bR^{n+1}, \bR_{\geq 0})$ and for $t \in (0,T)$. Also, recall that $\mu_t$ satisfies Conjecture \ref{Conjecture C} by Proposition \ref{proposition conjecture C holds for countable singular times}. So for a.e.\ $t \geq 0$, $\frac{d}{dt}\mu_t(\phi) = \cB(\mu_t, \phi)$. Then, for $0 < s,t < T$, applying absolute continuity,
$$\mu_t(\phi) - \mu_s(\phi) = \int_{s}^t \frac{d}{ds'}\mu_{s'}(\phi)\,ds' = \int \cB(\mu_{s'}, \phi)\,ds'$$
We conclude that $\mu_t$ satisfies Conjecture \ref{Conjecture D}, equality in the integral Brakke inequality, for $0 < s,t<T$. In fact, since $M_0$ is initially smooth and closed, the Brakke inequality holds for $0\leq s,t <T$.

\noindent\textbf{Brakke flow assumption:}

Assume that $\mu_t$ is an integral Brakke flow which is unit regular, has bounded area ratios, and has a countable set of singular times for $t<T= T_{\mult}$. The same proof as above works by applying Lemma \ref{lemma absolutely continuous mass} and Proposition \ref{proposition conjecture C holds for countable singular times}.

\subsection{Proof of Corollary \ref{corollary two convex three convex BFE}}\hfill\\
\vspace{-.2in}

\noindent \textbf{Spherical singularities and neckpinches:}

Assume the level set flow of $M_0$ has only spherical singularities and neckpinches. By Choi-Haslhofer-Hershkovits-White~\cite{ChoiHaslhoferHershWhite19}, the level set flow of $M_0$ satisfies $T_{\disc} = \infty$. By Proposition \ref{proposition no Tmult for lsf with cylindrical singularities} and the definition of spherical singularities and neckpinches of a level set flow, the only tangent flows of $\mu_t$ are cylinders or static multiplicity one planes. In particular, $T_{\mult} = \infty$. Therefore, $\min(T_{\disc}, T_{\mult}) = \infty$. Proposition \ref{proposition examples of main assumption} then implies that $(M_0, \mu_t)$ satisfies the General Assumption for $t < \infty$ and the level set flow of $M_0$ has an open cocountable set of regular times for $t<\infty$. We then apply Theorem \ref{theorem main assumption BFE before tau} to find that $\mu_t$ satisfies Conjecture \ref{Conjecture D}, the integral Brakke equality, for $0 \leq s,t <\infty$.
\newpage
\noindent \textbf{Three-convex blow-up type:}

Assume the level set flow of $M_0$ has three-convex blow-up type. Proposition \ref{proposition examples of main assumption} implies that $(M_0, \mu_t)$ satisfies the General Assumption for $t<T_{\disc}$ and the level set flow of $M_0$ has an open cocountable set of regular times for $t<T_{\disc}$. By Proposition \ref{proposition no Tmult for lsf with cylindrical singularities} and the definition of three-convex blow-up type for the level set flow, the only tangent flows of $\mu_t$ for $t< T_{\disc}$ are static multiplicity one planes and round shrinking cylinders. In particular, $T_{\disc}\leq T_{\mult}$. An application of Theorem \ref{theorem main assumption BFE before tau} implies that $\mu_t$ satisfies Conjecture \ref{Conjecture D}, the integral Brakke equality, for $0 \leq s,t < \min(T_{\disc}, T_{\mult}) = T_{\disc}$.

\noindent \textbf{Mean convex neighborhoods in low dimensions:}

Assume the level set flow of $M_0$ has mean convex neighborhoods of singularities in $\bR^{n+1}$, $n+1 \leq 4$. By work of Hershkovits-White~\cite{Wh00}~\cite[Remark 5]{HershkovitsWhite17}, $T_{\disc} = \infty$ and $(M_0, \mu_t)$ satisfies the General Assumption for $t < \infty$. Since the flow has mean convex neighborhoods of singularities, the blow-ups are the same as that of mean convex flow, since White's theory works locally~\cite{Wh00}. The dimension bound implies the level set flow of $M_0$ has three-convex blow-up type for $t \in [0, \infty)$. We may conclude as in the argument above for three-convex blow-up type, which proves that $\mu_t$ satisfies Conjecture \ref{Conjecture D}, the integral Brakke equality, for $0 \leq s,t < \infty$.

\subsection{Proof of Corollary \ref{corollary limit flows of BFEs}}\hfill\\
\vspace{-.2in}

In order to prove Corollary \ref{corollary limit flows of BFEs}, we begin with a lemma on limit flows. This lemma states that the class of limits flows of a Brakke flow is closed under the operation of taking limit flows. This lemma is invoked in the proof of~\cite[Theorem 12.2]{Wh00}, but we provide a proof here. Recall that $\tau$ is the time function for the Brakke flow.

\begin{lem}\label{lemma limit flow of limit flow is limit flow}
Let $\{\mu_t\}_{t \geq 0}$ be an integral Brakke flow with bounded area ratios. Let $\mu'_t$ be a limit flow of $\mu_t$ at the spacetime point $X$, $\tau(X)>0$. 

If $\mu''_t$ is a limit flow of $\mu'_t$ at any spacetime point $Y$, then $\mu''_t$ is a limit flow of $\mu_t$ at $X$.
\end{lem}
\begin{proof}
By assumption, there exists a sequence $\la_i \to \infty$ and $(x_i, t_i) = X_i \to X = (x,t)$ such that
\begin{equation}\label{equation convergence of limit flow}
\mu_t^{\la_i, X_i}(E):=\la_i^n \mu_{\la_i^{-2}t + t_i}(\la_i^{-1}E+x_i)\end{equation}
converges as Brakke flows to the limit flow $\mu'_t$. Since $\mu_t$ has bounded area ratios, $\mu'_t$ also has bounded area ratios. By construction $\mu'_t$ is defined for $t \in (-\infty, \infty)$. We may then take a limit flow $\mu''_t$ of $\mu'_t$ at any spacetime point $Y$. Let $\la_j \to \infty$ and $(y_j, t_j) = Y_j \to Y = (y,t)$ such that 
\begin{equation}\label{equation convergence of limit flow 2}{\mu'_t}^{\la_j, Y_j}(E):=\la_j^n \mu'_{\la_j^{-2}t + t_j}(\la_j^{-1}E+y_j)\end{equation} 
converges to the limit flow $\mu''_t$. Now, let $Z_{ij} :=(\la_i^{-1}y_j + x_i, \la_i^{-2}t_j + t_i)$. By construction, 
$$\lim_{(i,j)\to (\infty, \infty)}Z_{ij} = X$$
Then, there exists a subsequence of $(i,j) \to (\infty, \infty)$,  relabeled by $i,j$, such that
$$\mu_t^{\la_j\la_i,\, Z_{ij}}$$
converges as $(i,j) \to (\infty, \infty)$ to a Brakke flow $\tilde{\mu}_t$. Note that we are using the rescaling notation of (\ref{equation convergence of limit flow}). We may find this subsequence by applying the Brakke compactness theorem and using boundedness of area ratios (see Section \ref{section singularities mean curvature flow}). The goal is to show that $\tilde{\mu}_t = \mu''_t$. 

For each measurable $E \subset \bR^{n+k}$ and each $j \geq 1$,
\begin{align*}\label{equation first limit exists pointwise}
    \lim_{i \to \infty} \mu_t^{\la_j\la_i,\,Z_{ij}}(E) &= \la_j^n \lim_{i \to \infty} \la_i^n \mu_{\la_i^{-2}(\la_j^{-2}t+t_j)+t_i}(\la_i^{-1}(\la_j^{-1}E + y_j) + x_i)\\
    &= \la_j^n \mu'_{\la_j^{-2}t + t_j}(\la_j^{-1}E+y_j)
\end{align*}
where this limit exists by the fact that (\ref{equation convergence of limit flow}) converges to $\mu'_t$ as $\la_i \to \infty$. Then, by the fact that (\ref{equation convergence of limit flow 2}) converges to $\mu''_t$,
\begin{equation*}\label{iterated limit limit flows}
     \lim_{j \to \infty} \lim_{i \to \infty} \mu_t^{\la_j\la_i,\,Z_{ij}}(E) = \mu''_t(E) 
\end{equation*}
We recall a well-known lemma about double limits: if $\lim_{(i,j) \to (\infty, \infty)} f(i,j)$ exists and if for each $j$, $\lim_{i \to \infty}f(i,j)$ exists, then 
$$\lim_{(i,j) \to (\infty, \infty)} f(i,j) = \lim_{j \to \infty}\lim_{i \to \infty} f(i,j)$$
We have verified these conditions for limits of $\mu_t^{\la_j\la_i, \, Z_{ij}}(E)$ for each $E$, so we conclude that for each $t$,
$$\tilde{\mu}_t = \lim_{(i,j) \to (\infty, \infty)} \mu_t^{\la_j \la_i, \, Z_{ij}} = \lim_{j \to \infty} \lim_{i \to \infty} \mu_t^{\la_j\la_i,\,Z_{ij}} = \mu''_t$$
This shows that the limit flow $\tilde{\mu}_t$ equals $\mu''_t$ as Radon measures. Convergence as Brakke flows also requires convergence as varifolds at almost every time. The same argument works to show varifold convergence of $\mu_t^{\la_j \la_i, \, Z_{ij}}$ to $\mu''_t$ at almost every $t$. This proves that the limit flow $\mu''_t$ of $\mu'_t$ is a limit flow of $\mu_t$.
\end{proof}

Under the assumptions of Corollary \ref{corollary limit flows of BFEs}, let $\mu'_t$ be a limit flow at a spacetime point $X$, with $0<\tau(X)<T$. Recall the definition of $T$ in the statement of Corollary \ref{corollary limit flows of BFEs}. By Lemma \ref{lemma limit flow of limit flow is limit flow}, no tangent flow of $\mu'_t$ is a static or quasistatic plane with multiplicity $\geq 2$. Indeed, if such a tangent flow existed, then by Lemma \ref{lemma limit flow of limit flow is limit flow}, there would be a limit flow at $X$ which is a higher multiplicity plane. By assumption, we have that $T< T_{\mult}^*$, so no limit flow at $X$ is a higher multiplicity plane. This implies that $T_{\mult} = \infty$ for $\mu'_t$, i.e.\ no tangent flows of $\mu'_t$ are planes of multiplicity $\geq 2$. 

Also, $\mu_t$ is unit regular. In the case that $(M_0, \mu_t)$ satisfies the General Assumption for $t<T \leq T_{\disc}$, this follows from Proposition \ref{proposition no quasistatic mult 1 planes}. The limit of unit regular flows is unit regular~\cite[Appendix B]{HershkovitsWhite17}~\cite[Section 7]{White05}. We are taking a limit flow at $\tau(X)<T$, and $\mu_t$ is unit regular for $t<T$. So, $\mu'_t$ is unit regular as well. 

Finally, $\mu_t$ has bounded area ratios, so $\mu'_t$ has bounded area ratios as well.

Recall that $\mu'_t$ is defined for $t \in (-\infty, \infty)$. We have found that $\mu'_t$ is unit regular, has bounded area ratios, and has no planar tangent flows of multiplicity $\geq 2$, i.e. $T_{\mult}=\infty$. In particular, $\cH^n(\cT^+_{\mult}(t))=0$ for $t < T_{\mult}=\infty$. By Theorem \ref{theorem main result without unit regularity}, $t \mapsto \mu_t'(\phi)$ is continuous for $t \in (-\infty, \infty)$. Note that Theorem \ref{theorem main result without unit regularity} is written for flows with initial time $t=0$, but this can be applied for any initial time. 

Thus, for a limit flow $\mu'_t$ at $X$ with $\tau(X)<T$, $t \mapsto \mu_t(\phi)$ is continuous for $t \in (-\infty, \infty)$ and each $\phi \in C^2_c(\bR^{n+k}, \bR_{\geq 0})$.

Suppose in addition that $\mu_t'$ has a countable set of singular times in $(-\infty, \infty)$. Then, we may apply Theorem \ref{theorem main assumption BFE before tau} so that for $\phi \in C^2_c(\bR^{n+k}, \bR_{\geq 0})$ and $-\infty < s,t <\infty$,
    $$\mu'_{t}(\phi) - \mu'_{s}(\phi) = \int_{s}^t \int -\phi |\vec{H}_{\mu'_{s'}}|^2 + \na \phi \cdot \vec{H}_{\mu'_{s'}} \,d\mu'_{s'}\,ds'$$
This completes the proof of Corollary \ref{corollary limit flows of BFEs}.

\bibliographystyle{amsplain}
\bibliography{bibliography}

\end{document}